\documentclass[journal,twoside,web]{ieeecolor}
\usepackage{generic}
\usepackage{cite}
\usepackage{amsmath,amssymb,amsfonts}
\usepackage{graphicx}
\usepackage{textcomp}
\usepackage[absolute,overlay]{textpos}

\usepackage{amsthm}
\usepackage[caption=false]{subfig}
\usepackage{url}
\usepackage{dsfont}
\usepackage{wrapfig}
\usepackage{tikz}
\usetikzlibrary{automata, positioning, arrows}
\tikzset{
	->, 
	>=stealth', 
	initial text=$ $, 
}
\usepackage{algorithm}
\usepackage{algpseudocode}

\newtheorem{definition}{Definition}
\newtheorem{remark}{Remark}
\newtheorem{problem}{Problem}
\newtheorem{lemma}{Lemma}
\newtheorem{assumption}{Assumption}
\newtheorem{proposition}{Proposition}

\def \agent {\mathcal{A}}
\def \svisor {\mathcal{S}}
\def \states {S}
\def \actions {A}
\def \probs {P}
\def \mdp {\mathcal{M}}
\def \initialstate {s_{0}}
\def \atomicprops {AP}

\def \policy {\pi}
\def \Policy {\Pi}
\def \successor {Succ}
\def \specification {\phi}
\def \agentspec {\lozenge R^{\agent}}
\def \agentthr {\nu^{\agent}}
\def \svisorspeci {\lozenge R^{\svisor}_{i}}
\def \svisorthri {\nu^{\svisor}_{i}}
\def \svisornumofspec {K^{\svisor}}
\def \history {h}
\def \histories {\mathcal{H}}

{\tiny }

\def \closedset {C^{cl}}
\def \agentsset {R^{\agent}}

\def \differset {\states_{d}}

\graphicspath{{images/}}

\def\BibTeX{{\rm B\kern-.05em{\sc i\kern-.025em b}\kern-.08em
    T\kern-.1667em\lower.7ex\hbox{E}\kern-.125emX}}
\begin{document}
\title{Deception in Supervisory Control}
\author{\thanks{This work was supported in part by AFRL FA9550-19-1-0169, AFOSR FA9550-19-1-0005, and DARPA D19AP00004.}Mustafa O. Karabag, Melkior Ornik,
	and Ufuk Topcu
	\thanks{M. O. Karabag is with the Department of Electrical and Computer Engineering, The University of Texas at Austin, Austin, TX 78705 USA (e-mail:karabag@utexas.edu). }
	\thanks{M. Ornik was with The University of Texas at Austin, Austin, TX 78705 USA. He is now with the Department of Aerospace Engineering and the Coordinated Science Laboratory, The University of Illinois at Urbana-Champaign, Urbana, IL 61801, USA (e-mail:mornik@illinois.edu).}
	\thanks{U. Topcu is with the Department of Aerospace Engineering and Engineering Mechanics and the Institute for Computational Engineering and Sciences, The University of Texas at Austin, Austin, TX 78705 USA (e-mail:utopcu@utexas.edu).}}

\maketitle

\begin{textblock*}{20.5cm}(0.5cm,0.1cm) 
   \footnotesize \copyright 2021 IEEE.  Personal use of this material is permitted.  Permission from IEEE must be obtained for all other uses, in any current or future media, including eprinting/republishing this material for advertising or promotional purposes, creating new collective works, for resale or redistribution to servers or lists, or reuse of any copyrighted component of this work in other works.
\end{textblock*}

\begin{abstract}
The use of deceptive strategies is important for an agent that attempts not to reveal his intentions in an adversarial environment. We consider a setting in which a supervisor provides a reference policy and expects an agent to follow the reference policy and perform a task. The agent may instead follow a different, deceptive policy to achieve a different task. We model the environment and the behavior of the agent with a Markov decision process, represent the tasks of the agent and the supervisor with reachability specifications, and study the synthesis of optimal deceptive policies for such agents. We also study the synthesis of optimal reference policies that prevent deceptive strategies of the agent and achieve the supervisor's task with high probability. We show that the synthesis of optimal deceptive policies has a convex optimization problem formulation, while the synthesis of optimal reference policies requires solving a nonconvex optimization problem. We also show that the synthesis of optimal reference policies is NP-hard.
\end{abstract}

\begin{IEEEkeywords}
Markov decision processes, deception, supervisory control, computational complexity.
\end{IEEEkeywords}

\section{Introduction}

\IEEEPARstart{D}{eception} is present in many fields that involve two parties, at least one of which is performing a task that is undesirable to the other party. The examples include cyber systems~\cite{carroll2011game , almeshekah2016cyber}, autonomous vehicles~\cite{mceneaney2005deception}, warfare strategy~\cite{lloyd2003art}, and robotics~\cite{shim2013taxonomy}. We consider a setting with a supervisor and an agent where the supervisor provides a reference policy to the agent and expects the agent to achieve a task by following the reference policy. However, the agent aims to achieve another task that is potentially malicious towards the supervisor and follows a different, deceptive policy. We study the synthesis of deceptive policies for such agents and the synthesis of reference policies for such supervisors that try to prevent deception besides achieving a task. 

In the described supervisory control setting, the agent's deceptive policy is misleading in the sense that the agent follows his own policy, but convinces the supervisor that he follows the reference policy. Misleading acts result in plausibly deniable outcomes~\cite{doodylying}. Hence, the agent's misleading behavior should have plausible outcomes for the supervisor. In detail, the supervisor has an expectation of the probabilities of the possible events. The agent should manipulate these probabilities such that he achieves his task while closely adhering to the supervisor's expectations. 

We measure the closeness between the reference policy and the agent's policy by Kullback--Leibler (KL) divergence. KL divergence, also called relative entropy, is a measure of dissimilarity between two probability distributions~\cite{cover2012elements}. KL divergence quantifies the extra information needed to encode a posterior distribution using the information of a given prior distribution. We remark that this interpretation matches the definition of plausibility: The posterior distribution is plausible if the KL divergence between the distributions is low.  

We use a Markov decision process (MDP) to represent the stochastic environment and reachability specifications to represent the supervisor's and the agent's tasks. We formulate the synthesis of optimal deceptive policies as an optimization problem that minimizes the KL divergence between the distributions of paths under the agent's policy and reference policy subject to the agent's task specification. In order to preempt the agent's deceptive policies, the supervisor may aim to design its reference policy such that any deviations from the reference policy that achieves some malicious task do not have a plausible explanation. We formulate the synthesis of optimal reference policies as a maximin optimization problem where the supervisor's optimal policy is the one that maximizes the KL divergence between itself and the agent's deceptive policy subject to the supervisor's task constraints. 

The agent's problem, the synthesis of optimal deceptive policies, and the supervisor's problem, the synthesis of optimal reference policies, lead to the following questions: Is it computationally tractable to synthesize an optimal deceptive policy? Is it computationally tractable to synthesize an optimal reference policy? We show that given the supervisor's policy, the agent's problem reduces to a convex optimization problem, which can be solved efficiently. On the other hand, the supervisor's problem results in a nonconvex optimization problem even when the agent uses a predetermined policy. We show that the supervisor's problem is NP-hard. We propose the duality approach and alternating direction method of multipliers (ADMM)~\cite{boyd2011distributed} to locally solve the supervisor's optimization problem. We also give a relaxation of the problem that is a linear program.

The setting described in this paper can be considered as a probabilistic discrete event system under probabilistic supervisory control~\cite{pantelic2009probabilistic,lawford1995equivalence}. The probabilistic supervisor induces an explicit probability distribution over the language generated by the system by random disablement of the events. The supervisory control model considered in this paper is similar in that the reference policy induces an explicit probability distribution over the paths of the MDP. Different from \cite{pantelic2009probabilistic,lawford1995equivalence},  we consider that the random disablement is done by the agent, and the supervisor is only responsible for providing the explicit random disablement strategy.

Similar to our approach,~\cite{bakshi2018plausible} used KL divergence as a proxy for the plausibility of messages in broadcast channels. While we use the KL divergence for the same purpose, the context of this paper differs from~\cite{bakshi2018plausible}. In the context of transition systems,~\cite{boularias2011relative, levine2014learning} used the metric proposed in this paper, the KL divergence between the distributions of paths under the agent's policy and the reference policy, for inverse reinforcement learning. In addition to the contextual difference, the proposed method of this paper differs from~\cite{boularias2011relative, levine2014learning}. We work in a setting with known transition dynamics and provide a convex optimization problem to synthesize the optimal policy while~\cite{boularias2011relative, levine2014learning} work with unknown dynamics and use sampling-based gradient descent to synthesize the optimal policy. Entropy maximization for MDPs~\cite{savas2018entropy} is a special case of the deception problem where the reference policy follows every possible path with equal probability. One can synthesize optimal deceptive policies by maximizing the entropy of the agent's path distribution minus the cross-entropy of the supervisor's path distribution relative to the agent's. For the synthesis of optimal deceptive policies, we use a method similar to~\cite{savas2018entropy} as we represent the objective function over transition probabilities. However, our proofs for the existence and synthesis of the optimal deceptive policies significantly differ from the results of \cite{savas2018entropy}. In particular, \cite{savas2018entropy} restricts attention to stationary policies without optimality guarantees whereas we prove the optimality of stationary policies for the deception problem. A related concept to deception is probabilistic system opacity, which was introduced in \cite{keroglou2018probabilistic}. Two hidden Markov models (HMMs) are pairwise probabilistically opaque if the likelihood ratio between the HMMs is a non-zero finite number for every infinite observation sequence. This paper is related to \cite{keroglou2018probabilistic} in that two HMMs are not pairwise probabilistically opaque if the KL divergence between the distribution of observation sequences is infinite. While \cite{keroglou2018probabilistic} provides a method to check whether two HMMs with irreducible Markov chains are pairwise probabilistically opaque, we propose an optimization problem for MDPs that quantifies the deceptiveness of the induced system by the agent. Deception is also interpreted as the exploitation of an adversary's inaccurate beliefs on the agent's behavior~\cite{ornik2018deception,karabag2019least}. The work \cite{ornik2018deception} focuses on generating unexpected behavior conflicting with the beliefs of the adversary, and \cite{karabag2019least} focuses on generating noninferable behavior leading to inaccurate belief distributions. On the other hand, the deceptive policy that we present generates behavior that is closest to the beliefs of the other party in order to hide the agent's malicious intentions. 

 We explore the synthesis of optimal reference policies, which, to the best of our knowledge, has not been discussed before. We propose to use ADMM to synthesize the optimal reference policies. Similarly,~\cite{fu2015optimal} also used ADMM for the synthesis of optimal policies for MDPs. While we use the same method, the objective functions of these papers differ since~\cite{fu2015optimal} is concerned with the average reward case, whereas we use ADMM to optimize the KL divergence between the distributions of paths. In addition, we use the ADMM to solve a decomposable minimax problem, which, to the best of our knowledge, has not been explored before.

We remark that a preliminary conference version \cite{karabag2019optimal} of this paper focuses on the synthesis of deceptive policies. In addition to contents of \cite{karabag2019optimal}, this version contains the NP-hardness result on the synthesis of optimal reference policies, duality and ADMM methods for the synthesis of locally optimal reference policies, and an additional numerical example~(Sections \ref{subsection:hardness}, \ref{subsection:dual}, \ref{subsection:ADMM}, and \ref{subsection:admmexample}). We also provide proofs of our results, which were omitted from \cite{karabag2019optimal}.
 
The rest of the paper is organized as follows. Section \ref{section:preliminaries} provides necessary theoretical background. In Section \ref{section:problemstatement}, the agent's and the supervisor's problems are presented. Section \ref{section:deceptive} explains the synthesis of optimal deceptive policies. In Section \ref{section:reference}, we give the NP-hardness result on the synthesis of optimal reference policies. We derive the optimization problem to synthesize the optimal reference policy and give the ADMM algorithm to solve the optimization problem. In this section, we also give a relaxed problem that relies on a linear program for the synthesis of optimal reference policies. We present numerical examples in Section \ref{section:examples} and conclude with suggestions for future work in Section \ref{section:conclusion}. We discuss the optimal deceptive policies under nondeterministic reference policies in Appendix \ref{section:nondeterministicref}. We provide the proofs for the technical results in Appendix \ref{section:proofs}.

	\section{Preliminaries} \label{section:preliminaries}
The set $\lbrace{ x=(x_{1}, \ldots, x_{n}) | x_i \geq 0 \rbrace}$ is denoted by $\mathbb{R}_{+}^{n}$. The set $\lbrace 1, \ldots, n \rbrace$ is denoted by $[n].$ The indicator function $\mathds{1}_{y}(x)$ of an element $y$ is defined as $\mathds{1}_{y}(x) =1$ if $x=y$ and $0$ otherwise. The characteristic function $\mathcal{I}_{C}(x)$ of a set $C$ is defined as $\mathcal{I}_{C}(x) = 0$ if $x \in C$ and $\infty$ otherwise. The projection $Proj_{C}(x)$ of a variable $x$  to a set $C$ is equal to $\arg\min_{y \in C} \| x-y \|^{2}_{2}$. A Bernoulli random variable with parameter $p$ is denoted by $Ber(p)$.

The set $\mathcal{K}$ is a convex cone, if for all $x,y \in \mathcal{K}$ and $a,b \geq 0$, we have $ax +by \in \mathcal{K}$. For the convex cone $\mathcal{K}$, $\mathcal{K}^{*} = \lbrace y | y^{T}x \geq 0, \forall x \in \mathcal{K} \rbrace$ denotes the dual cone. The exponential cone is denoted by $\mathcal{K}_{\exp} = \lbrace (x_1,x_2,x_3) | x_2 \exp(x_1 / x_2) \leq x_3, \ x_2 > 0 \rbrace$ $\cup$ $ \lbrace (x_1,0,x_3) | x_1 \leq 0, \ x_3 \geq 0 \rbrace$ and it can be shown that $\mathcal{K}^*_{\exp} = \lbrace (x_1,x_2,x_3) | -x_1 \exp(x_2 / x_1 - 1) \leq x_3, \ x_1 < 0 \rbrace$ $\cup$ $ \lbrace (0,x_2,x_3) | x_2 \geq 0, \ x_3 \geq 0 \rbrace$. 

\begin{definition}
	Let $Q_1$ and $Q_2$ be discrete probability distributions with a countable support $\mathcal{X}$. The \textit{Kullback--Leibler divergence} between $Q_1$ and $Q_2$ is $$KL(Q_1 || Q_2) = \sum_{x \in \mathcal{X}} Q_1(x) \log \left( \frac{Q_1(x)}{Q_2(x)}\right).$$ 
\end{definition}
We define $Q_1(x) \log \left( \frac{Q_1(x)}{Q_2(x)} \right)$ to be $0$ if $Q_1(x)=0$, and $\infty$ if $Q_1(x)> 0$ and $Q_2(x)=0$. 
Data processing inequality states that any transformation $T:\mathcal{X} \to \mathcal{Y}$ satisfies \begin{equation} \label{ineq:dataprocessing}
	KL(Q_{1} || Q_{2}) \geq KL(T(Q_{1}) || T(Q_{2})).
\end{equation}

\begin{remark}
	KL divergence is frequently defined with logarithm to base 2 in information theory. However, we use natural logarithm for the clarity of representation in the optimization problems. The base change does not change the results.
\end{remark}

	\subsection{Markov Decision Processes}
		A \textit{Markov decision process} (MDP) is a tuple $\mdp = (\states, \actions, \probs, \initialstate)$ where $\states$ is a finite set of states, $\actions$ is a finite set of actions, $\probs: \states \times \actions \times \states \to [0,1]$ is the transition probability function, and  $\initialstate$ is the initial state. $\actions(s)$ denotes the set of available actions at state $s$ where $\sum_{q \in \states} \probs(s,a,q) = 1$ for all $a \in \actions(s)$. The successor states of state $s$ is denoted by $Succ(s)$ where a state $q$ is in $Succ(s)$ if and only if there exists an action $a$ such that $P(s,a,q)> 0$. State $s$ is \textit{absorbing} if $\probs(s,a,s) = 1$ for all $a \in \actions(s)$.
	
		 The \textit{history} $\history_{t}$ at time $t$ is a sequence of states and actions such that $\history_{t} = s_{0} a_{0} s_{1} \ldots s_{t-1} a_{t-1} s_{t}$. The set of all histories at time $t$ is $\histories_{t}$. A \textit{policy} for $\mdp$ is a sequence $\policy = \mu_0\mu_1\ldots$ where each $\mu_t: \histories_{t} \times \actions \to [0,1]$ is a function such that $\sum_{a \in \actions(s_{t})} \mu_t(\history_{t},a) = 1$ for all $\history_{t} \in \histories_{t}$.   A \textit{stationary policy} is a sequence $\policy = \mu \mu \ldots$ where $\mu:\states \times \actions \to [0,1]$ is a function such that $\sum_{a \in \actions(s)} \mu(s,a) = 1$ for every $s \in \states$. The set of all policies for $\mdp$ is denoted by $\Pi(\mdp)$ and the set of all stationary policies for $\mdp$ is denoted by $\Pi^{St}(\mdp)$. For notational simplicity, we use $\probs_{s,a,q}$ for $\probs(s,a,q)$ and $\policy_{s,a}$ for $\mu(s,a)$ if $\pi=\mu\mu\ldots$, i.e., $\pi$ is stationary.

	A stationary policy $\policy$ for $\mdp$ induces a Markov chain $\mdp^{\policy} = (\states, \probs^{\policy})$ where $\states$ is the finite set of states and $\probs^{\policy}:\states \times \states \to [0,1]$ is the transition probability function such that $\probs^{\policy}(s,q) = \sum_{a \in A(s)} \probs(s,a,q) \policy(s,a)$ for all $s,q \in \states$. A state $q$ is \textit{accessible} from a state $s$ if there exists an $n \geq 0$ such that the probability of reaching $q$ from $s$ in $n$ steps is greater than $0$ . A set $C$ of states is a \textit{communicating class} if $q$ is accessible from $s$, and $s$ is accessible from $q$ for all $s,q \in C$. A communicating class $C$ is \textit{closed} if $q$ is not accessible from $s$ for all $s\in C$ and $q \in \states \setminus C$. 

		A \textit{path} $\xi = s_0 s_1 s_2 \ldots$ for an MDP $\mdp$ is an infinite sequence of states under policy $\policy= \mu_0 \mu_1 \ldots$ such that $\sum_{a \in \actions(s_t)} \probs(s_t,a,s_{t+1}) \mu_t(h_t,a) > 0$ for all $t\geq0$.  The distribution of paths for $\mdp$ under policy $\policy$ is denoted by $\Gamma^{\policy}_{\mdp}$.

	For an MDP $\mdp$ and a policy $\policy$, the \textit{state-action occupation measure} at state $s$ and action $a$ is defined by $x^{\policy}_{s,a} := \sum_{t=0}^{\infty} \Pr (s_t = s | s_0) \mu_t(s_t,a).$ If $\policy$ is stationary, the state-action occupation measures satisfy $x^{\policy}_{s,a} = \pi_{s,a} \sum_{a' \in A(s)} x^{\policy}_{s,a'}$ for all $s$ with finite occupation measures. The state-action occupation measure of a state-action pair is the expected number of times that the action is taken at the state over a path. We use $x^{\policy}_{s}$ for the vector of the state-action occupation measures at state $s$ under policy $\policy$ and $x^{\pi}$ for the vector of all state-action occupation measures. 

	We use $\lozenge R$ to denote the reachability specification to set $R$. A path $\xi = s_0 s_1 s_2 \ldots$ satisfies $\lozenge R$  if and only if there exists $i$ such that $s_{i} \in R$. On an MDP $\mdp$, the probability that a specification  $\lozenge R$ is satisfied under a policy $\policy$, is denoted by $\Pr^{\policy}_{\mdp}(\initialstate \models \lozenge R)$.

	\section{Problem Statement} \label{section:problemstatement}
	We consider a setting in which an agent operates in a discrete stochastic environment modeled with an MDP $\mdp$, and a supervisor provides a reference policy $\policy^{\svisor}$ to the agent. The supervisor expects the agent to follow $\policy^{\svisor}$ on $\mdp$, thereby performing $\svisornumofspec$ tasks that are specified by reachability specifications $\svisorspeci$ for all $i \in [\svisornumofspec]$. The agent aims to perform another task that is specified by the reachability specification $\agentspec$ and may deviate from the reference policy to follow a different policy $\policy^{\agent}$. In this setting, both the agent and the supervisor know the environment, i.e., the components of $\mdp$.

	While the agent operates in $\mdp$, the supervisor observes the transitions, but not the actions of the agent, to detect any deviations from the reference policy. An agent that does not want to be detected must use a deceptive policy $\policy^{\agent}$ that limits the amount of deviations from reference policy $\policy^{\svisor}$ and achieves $\agentspec$ with high probability.
	
	We use Kullback-Leibler (KL) divergence to measure the deviation from the supervisor's policy. Recall that $\Gamma^{\policy^{\svisor}}_{\mdp}$ and $\Gamma^{\policy^{\agent}}_{\mdp}$ are the distributions of paths under $\policy^{\svisor}$ and $\policy^{\agent}$, respectively. We consider $KL(\Gamma^{\policy^{\agent}}_{\mdp} || \Gamma^{\policy^{\svisor}}_{\mdp})$ as a proxy for the agent's deviations from the reference policy. 
	
	The perspective of information theory provides two motivations for the choice of KL divergence. The obvious motivation is that this value corresponds to the amount of information bits that the reference policy lacks while encoding the agent's path distribution. By limiting the deviations from the reference policy, we aim to make the agent's behavior easily explainable by the reference policy. Sanov's theorem~\cite{cover2012elements} provides the second motivation. We note that satisfying the agent's objective with high probability is a rare event under the supervisor's policy. By minimizing the KL divergence between the policies, we make the agent's policy mimic the rare event that satisfies the agent's objective and is most probable under the supervisor's policy. Formally, let $\policy^{*}$ be a solution to 
		\begin{align*}
		\underset{\policy \in \Policy(\mdp)}{\inf} \quad
		&  KL\left(\Gamma^{\policy}_{\mdp} || \Gamma^{\policy^{\svisor}}_{\mdp} \right) 
		\\
		\text{\normalfont subject to } \quad
		& {\Pr }^{\policy}_{\mdp}(s_0 \models \agentspec) \geq \nu^{\agent}. 
		\end{align*}
 Assume that we simulate $n$ paths under the supervisor's policy. The probability that the observed paths satisfy $\agentspec$ with probability higher than $\agentthr$ is approximately equal to $\exp({-n KL(\Gamma^{\policy^{*}}_{\mdp} || \Gamma^{\policy^{\svisor}}_{\mdp})})$. Furthermore, given that the observed path distribution satisfies $\agentspec$ with a probability higher than $\agentthr$, the most likely distribution is $\Gamma^{\policy^{*}}_{\mdp}$~\cite{cover2012elements}.

		The choice of KL divergence is also justified from the perspective of statistical hypothesis testing. Likelihood-ratio test, the most powerful test for a given significance level~\cite{neyman1933ix}, is a hypothesis test method to compare two models according to their goodness in terms of fitting the data. Assume that the supervisor uses likelihood-ratio test to decide whether the agent followed the reference policy or a different policy. Let $\xi_{1}, \ldots, \xi_{n}$ be the agent's paths that are observed by the supervisor. We denote the probability of $\xi_{1}, \ldots, \xi_{n}$ under $\policy^{\svisor}$ and $\policy^{\agent}$ by $\Pr(\xi_{1}, \ldots, \xi_{n} | \policy^{\svisor})$ and  $\Pr(\xi_{1}, \ldots, \xi_{n} | \policy^{\agent})$, respectively. By the likelihood-ratio test, the supervisor decides that the agent uses $\policy^{\agent}$ if $\log(\Pr(\xi_{1}, \ldots, \xi_{n} | \policy^{\agent})) - \log(\Pr(\xi_{1}, \ldots, \xi_{n} | \policy^{\svisor})) \geq c$, and the agent uses $\policy^{\svisor}$ otherwise where $c$ is a constant. We note that $\policy^{*}$ minimizes $n KL(\Gamma^{\policy}_{\mdp} || \Gamma^{\policy^{\svisor}}_{\mdp}) = \mathbb{E} \left[ \log(\Pr(\xi_{1}, \ldots, \xi_{n} | \policy^{\agent})) - \log(\Pr(\xi_{1}, \ldots, \xi_{n} | \policy^{\svisor})) \right]$ subject to ${\Pr }^{\policy}_{\mdp}(s_0 \models \agentspec) \geq \agentthr$. Therefore, in expectation $\policy^{*}$ is most likely policy to be not detected by the supervisor among the policies that satisfy ${\Pr }^{\policy}_{\mdp}(s_0 \models \agentspec) \geq \agentthr$.

	We propose the following problem for the synthesis of deceptive policies for the agents.
	\begin{problem}[Synthesis of Optimal Deceptive Policies]
		\label{problem:mindeviation}
		Given an MDP $\mdp$, a reachability specification $\agentspec$, a probability threshold $\agentthr$, and a reference policy $\policy^{\svisor}$, solve
		\begin{subequations}
			\label{problemeqn:mindeviation}
			\begin{align}
			\underset{\policy^{\agent} \in \Policy(\mdp)}{\inf} \quad
			 &  KL\left(\Gamma^{\policy^{\agent}}_{\mdp} || \Gamma^{\policy^{\svisor}}_{\mdp} \right) \label{eq:pathprobmininf}
			\\
			 \text{\normalfont subject to } \quad
			 & {\Pr }^{\policy^{\agent}}_{\mdp}(\initialstate \models \agentspec) \geq \agentthr. \label{cons:ourtask}
			\end{align}
		\end{subequations} If the optimal value is attainable, find a policy $\policy^{\agent}$ that is a solution to \eqref{problemeqn:mindeviation}.
	\end{problem}
	
	In order to preempt the possibility of that the agent uses a policy $\policy^{\agent}$ that is the best deceptive policy against $\policy^{\svisor}$, the supervisor aims to find a reference policy $\policy^{\svisor}$ that maximizes the divergence between $\policy^{\agent}$ and $\policy^{\svisor}$ subject to ${\Pr}^{\policy^{\svisor}}_{\mdp}(\initialstate \models \svisorspeci) \geq \svisorthri$ for all $i \in [\svisornumofspec]$. We assume that the supervisor knows the agent's task and propose the following problem for the synthesis of reference policies for the supervisor.
	\begin{problem}[Synthesis of Optimal Reference Policies] 
		\label{problem:mindeviation2}
		Given an MDP $\mdp$, reachability specifications $\agentspec$ and $\svisorspeci$ for all $i \in [\svisornumofspec]$, probability thresholds $\agentthr$ and $\svisorthri$ for all $i \in [\svisornumofspec]$, solve
		\begin{subequations} \label{problemeqn:mindeviation2}
			\begin{align}
			\underset{ \policy^{\svisor} \in \Policy(\mdp)}{\sup} \ \underset{ \policy^{\agent} \in \Policy(\mdp)}{\inf} \quad
			 &  KL\left(\Gamma^{\policy^{\agent}}_{\mdp} || \Gamma^{\policy^{\svisor}}_{\mdp} \right) \label{eq:pathprobmininf2}
			\\
			 \text{\normalfont subject to } \quad
& {\Pr }^{\policy^{\agent}}_{\mdp}(\initialstate \models \agentspec) \geq \agentthr,
			\\
			  \quad
& {\Pr }^{\policy^{\svisor}}_{\mdp}(\initialstate \models \svisorspeci) \geq \svisorthri, &&  \forall i \in [\svisornumofspec]. \label{cons:supervisorstask2}
			\end{align}
		\end{subequations}
	If the supremum is attainable, find a policy $\policy^{\svisor}$ that is a solution to \eqref{problemeqn:mindeviation2}.
	\end{problem}

			\begin{wrapfigure}{r}{0.23\textwidth}
	\centering
	\resizebox{0.2\textwidth}{!}{ \centering
		\begin{tikzpicture}
		\node[state, initial]  (s0) {$s_0$};
		\node[state] [right=of s0] (s2) {$s_2$};
		\node[state] [above=of s2] (s1) {$s_1$};
		\node[state] [below=of s2] (s3) {$s_3$};
		\draw 
		(s0) edge[bend left=40] node[above, sloped] {$\alpha,1$} (s1)
		(s0) edge[bend right=40] node[below, sloped] {$\gamma,1$} (s3)
		(s0) edge[bend right=0] node[below, sloped] {$\beta,0.1$} (s1)
		(s0) edge node[below, sloped] {$\beta,0.8$} (s2)
		(s0) edge node[below, sloped] {$\beta,0.1$} (s3)
		(s1) edge[loop right] node[right] {$\alpha,1$} (s1)
		(s2) edge[loop right] node[right] {$\alpha,1$} (s2)
		(s3) edge[loop right] node[right] {$\alpha,1$} (s3);
		\end{tikzpicture}}
	\caption{An MDP with 4 states. A label $a,p$ of a transition refers to the transition that happens with probability $p$ when action $a$ is taken.}
	\label{fig:basicexample}
\end{wrapfigure}

	\textbf{Example 1:}
		We explain the synthesis of optimal deceptive policies and reference policies through the MDP $\mdp$ given in Figure \ref{fig:basicexample}. Note that the policies for $\mdp$ may vary only at $s_{0}$ since it is the only state with more than one action. 
		
		We first consider the synthesis of optimal deceptive policies where the reference policy satisfies $\policy^{\svisor}_{s_{0},\beta}= 1$. Consider $\agentspec = \lozenge \lbrace s_{3} \rbrace$ and $\agentthr = 0.2$. Assume that the agent's policy has $\policy^{\agent}_{s,\gamma} = 1$. The value of the KL divergence is $2.30$. However, note that as $\policy^{\agent}_{s,\beta}$ increases, the KL divergence decreases. In this case, the optimal policy  satisfies $\policy^{\agent}_{s,\beta} = 0.89$ and $\policy^{\agent}_{s,\gamma} = 0.11$ and the optimal value for the KL divergence is $0.04$.
		 
		We now consider the synthesis of optimal reference policies where the supervisor has a single specification $\lozenge R^{\svisor}  = \lozenge \lbrace s_{1},  s_{2}\rbrace $ and $\nu^{\svisor} = 0.9$. Consider $\agentspec = \lozenge \lbrace s_{3} \rbrace$ and $\agentthr = 0.1$. Assume that we have $\policy^{\svisor}_{s_{0}, \beta}=1$. In this case, the agent can directly follow the reference policy and make the KL divergence zero. This reference policy is not optimal; the supervisor, knowing the malicious objective of the agent, can choose the reference policy with $\policy^{\svisor}_{s_{0}, \alpha}=1$, which does not allow any deviations and makes the KL divergence infinite.

	\section{Synthesis of Optimal Deceptive Policies} \label{section:deceptive}
	In this section, we explain the synthesis of optimal deceptive policies. Before proceeding to the synthesis step, we make assumptions to simplify the problem. Then, we show the existence of an optimal deceptive policy and give an optimization problem to synthesize one.

	Without loss of generality, we make the following assumption on the target states of the agent and the supervisor for the clarity of representation. This assumption ensures that the probability of completing a task is constant, either $0$ or $1$, upon reaching a target state.
	\begin{assumption} \label{assumption:closednessanddisjointness}
	Every $s \in R^{\agent} \cup R^{\svisor}_{1} \cup \ldots \cup R^{\svisor}_{\svisornumofspec}$ is absorbing.
	\end{assumption}

	We remark that in the absence of Assumption \ref{assumption:closednessanddisjointness}, one can still find the optimal deceptive policy by constructing a product MDP that encodes both the state of the original MDP and the statuses of the tasks. In detail, we need to construct a joint deterministic finite automaton whose states encode the statuses of the specifications for the agent and the supervisor. After creating the joint deterministic finite automaton (DFA), we construct a product MDP by combining the original MDP and the joint DFA and synthesize a policy on the product state space. Since there is a one-to-one mapping between the paths of the original MDP and the product MDP, the synthesized policy for the product MDP can be translated into a policy  for the original MDP~\cite{baier2008principles}.

If the reference policy is not stationary, we may need to compute the optimal deceptive policy by considering the parameters of the reference policy at different time steps. Such computation leads to a state explosion, which we avoid by adopting the following assumption.
	\begin{assumption} \label{assumption:supervisorisstationary}
		The reference policy $\policy^{\svisor}$ is stationary on $\mdp$.
	\end{assumption}

	 In many applications the supervisor aims to achieve the specifications with the maximum possible probabilities. Under Assumption \ref{assumption:closednessanddisjointness}, stationary policies suffice to achieve the Pareto optimal curve for maximizing the probabilities of multiple reachability specifications~\cite{etessami2007multi}.

	Without loss of generality, we assume that the optimal value of Problem \ref{problem:mindeviation} is finite. One can easily check whether the optimal value is finite in the following way. Assume that the transition probability between a pair of states is zero under the reference policy. One can create a modified MDP from $\mdp$ by removing the actions that assign a positive value to such state-state pairs. If there exists a policy that satisfies the constraint \eqref{cons:ourtask} then the value is finite.

	Given that the optimal value of Problem \ref{problem:mindeviation} is finite, we first identify the three sets of states where the agent should follow the reference policy. Firstly, the agent's policy should not be different from the supervisor's policy on the states that belong to $R^{\agent}$, since the specification of the agent is already satisfied. Secondly, the agent should follow the reference policy at states that are recurrent under the reference policy. Formally, the reference policy $\policy^{\svisor}$ induces a Markov chain $\mdp^{\svisor}$. A state is recurrent in $\mdp^{\svisor}$ if it belongs to some closed communicating class. The agent should follow the reference policy if a state is recurrent in $\mdp^{\svisor}$.

  For the second claim, we first remark that every closed communicating class $C \subset \states$ of $\mdp^{\svisor}$ satisfy either $1)$ $C \cap (\states \setminus R^{\agent}) \neq \emptyset$ and $C \cap  R^{\agent} = \emptyset$, or $2)$ $C \cap (\states \setminus R^{\agent}) = \emptyset$ and $C \cap  R^{\agent} \neq \emptyset$. This is due to the fact that $R^{\agent}$ is a closed set, i.e., a state in $R^{\agent}$ is reached and the states in $\states \setminus R^{\agent}$ are not accessible. Hence, there cannot be a closed communicating class of $\mdp^{\svisor}$ that has states in both $ R^{\agent}$ and $\states \setminus R^{\agent}$ . Let $\closedset$ be the union of all closed communicating classes of $\mdp^{\svisor}$, i.e., the recurrent states of $\mdp^{\svisor}$. Note that $\closedset \setminus  R^{\agent}$ is a closed set in $\mdp^{\svisor}$ and the states in $R^{\agent}$ are not accessible from $\closedset \setminus  R^{\agent}$ in $\mdp^{\svisor}$ due to the above discussion.

Assume that under the agent's policy $\policy^{\agent}$, there exists a path that visits a state in $\closedset \setminus  R^{\agent}$ and leaves $\closedset \setminus  R^{\agent}$ with positive probability. In this case, the KL divergence is infinite since an event that happens with probability zero under the supervisor's policy happens with a positive probability under the agent's policy. Hence, $\closedset \setminus  R^{\agent}$ must also be a closed set under $\policy^{\agent}$. Furthermore, since the agent cannot leave $\closedset \setminus  R^{\agent}$, and the probability of satisfying $\lozenge R^{\agent}$ is zero upon entering $\closedset \setminus  R^{\agent}$, the agent should choose the same policy as the supervisor to minimize the KL divergence between the distributions of paths. Note that for the recurrent states in $R^{\agent}$, i.e., $\closedset \cap R^{\agent} $, the second claim is trivially satisfied by the first claim.

	For all $s \in \states \setminus(C^{cl} \cup R^{\agent})$, $s$ is transient in $\mdp^{\svisor}$, and the agent's policy must eventually stop visiting $s$, since otherwise we have infinite divergence. Furthermore, we have the following proposition.

	\begin{proposition}
		\label{proposition:expectedtimesarefinite}
		If the optimal value of Problem \ref{problem:mindeviation} is finite and the optimal policy is $\policy^{\agent}$, the state-action occupation measure $x^{\policy^{\agent}}_{s,a}$ is finite for all $s \in S \setminus(C^{cl} \cup R^{\agent})$ and $a \in A(s)$.
	\end{proposition}

	The occupation measures are bounded for the states that the agent's policy may differ from the supervisor's policy. Since the occupation measures are bounded, the stationary policies suffice for the synthesis of optimal deceptive policies~\cite{altman1999constrained}.
	\begin{proposition} \label{proposition:stationaryisenough}
		For any policy $\pi^{\agent} \in \Pi(\mdp)$ that satisfies ${\Pr }^{\policy^{\agent}}_{\mdp}(s_0 \models \agentspec) \geq \agentthr$, there exists a stationary policy $\policy^{\agent, St} \in \Pi(\mdp)$ that satisfies ${\Pr }^{\policy^{\agent, St}}_{\mdp}(s_0 \models \agentspec) \geq \agentthr$ and $$KL\left(\Gamma^{\policy^{\agent, St}}_{\mdp} || \Gamma^{\policy^{\svisor}}_{\mdp} \right) \leq KL\left(\Gamma^{\policy^{\agent}}_{\mdp} || \Gamma^{\policy^{\svisor}}_{\mdp} \right).$$
	\end{proposition}

	\begin{proof}[\textbf{Sketch of Proof for Proposition \ref{proposition:stationaryisenough}}]
	Assume that the KL divergence between the path distributions is finite. Note that the occupation measures of $\policy^{\agent}$ are finite for all $s \in \differset = \states \setminus(C^{cl} \cup R^{\agent})$. 
	
	When the reference policy is stationary, we may transform $\mdp$ into a \textit{semi-infinite MDP}. The semi-infinite MDP shares the same states with $\mdp$, but has continuous action space such that for all states every randomized action of $\mdp$ is an action of the semi-infinite MDP. Also the states belong to $\agentsset$ and $\closedset$ are absorbing in the semi-infinite MDP. 
	
	Let $X^{\svisor}_{s}$ be the successor state distribution at state $s$ under the reference policy in the semi-infinite MDP. At state $s \in \differset$, an action $a$ with successor state distribution $X_{s,a}$ has cost $KL(X_{s,a} || X^{\svisor}_{s})$. The cost is $0$ for the other states that do not belong to $\differset$. Consider an optimization problem that minimizes the expected cost subject to reaching $\agentsset$ with probability at least $\nu^{\agent}$. The result of this optimization problem shares the same value with the result of Problem \ref{problem:mindeviation}. This problem is a constrained cost minimization for an MDP where the only decision variables are the state-action occupation measures. An optimal policy can be characterized by the state-action occupation measures.
	
	The occupation measures must be finite for all $s \in \differset$ as we showed in Proposition \ref{proposition:expectedtimesarefinite}. Since every finite occupation measure vector of $\differset$ can also be achieved by a stationary policy, there exists a stationary policy which shares the same occupation measures with an optimal policy~\cite{altman1999constrained}. Hence, this stationary policy is also optimal. 
	
	Now assume that the stationary optimal policy $\policy^{*}$ is randomized. Let $\pi^{*}_{s}$ be the action distribution and $X^{\pi^{*}}_{s}$ be the successor state distribution at state $s$ under $\policy^{*}$. Note that at state $s$ there exists an action $a^*$ that has $P(s,a^*,q) = X^{\policy}_{s}(q)$ since the action space is convex for the semi-infinite MDP. Also due to the convexity of KL divergence we have $\int_{\Delta^{|A(s)|}} KL(X_{s,a} || X_{s}^{\mathcal{S}}) d \pi^{*}_{s}(a) \geq KL(X^{\pi^{*}}_{s} || X_{s}^{\mathcal{S}})$ where $\Delta^{|A(s)|}$ is $|A(s)|$-dimensional probability simplex. Hence, deterministically taking action $a^{*}$ is optimal for state $s$. By generalizing this argument to all $s\in S_{s}$, we conclude that there exists an optimal stationary deterministic policy for the semi-infinite MDP. Without loss of generality we assume $\policy^{*}$ is stationary deterministic.
	
	We note that the stationary deterministic policy $\policy^{*}$ of the semi-infinite MDP corresponds to a stationary randomized policy for the original MDP $\mdp$. Hence the proposition holds.
	
\end{proof}

	We denote the set of states for which the agent's policy can differ from the supervisor's policy by $\differset = \states \setminus (\closedset \cup \agentsset)$. We solve the following optimization problem to compute the  occupation measures of an optimal deceptive policy:
	\begin{subequations}
	\label{prog:stateprogram}
	\begin{align}
	&\inf \quad \sum_{s \in \differset } \sum_{a \in \actions(s)} \sum_{q \in \successor(s)}  x^{\agent}_{s,a} \probs_{s,a,q} \nonumber
	\\& \hspace{1cm}\log \left( \frac{\sum_{a' \in \actions(s)} x^{\agent}_{s,a'} \probs_{s,a',q}}{ \policy^{\svisor}_{s,q} \sum_{a' \in \actions(s)} x^{\agent}_{s,a'}} \right) \label{eq:mininfrestime}
	\\
	& \normalfont \text{subject to } \quad \nonumber
	\\
	& x^{\agent}_{s,a} \geq 0,  \hspace{3.4cm}
	\forall s \in \differset, \ 
	\forall a \in \actions(s), \label{cons:positiveactions}
	\\
	& \hspace{-0.1cm}  \sum_{a \in \actions(s)} x^{\agent}_{s,a} - \sum_{q \in \differset}  \sum_{a \in \actions(q)} x^{\agent}_{q,a}\probs_{q,a,s} = \mathds{1}_{s_0}(s), \forall s \in \differset, \label{cons:floweqn}  
	\\
	& \sum_{q \in R^{\agent}} \sum_{s \in \states_{d}}  \sum_{a \in \actions(s)}  x^{\agent}_{s,a}\probs_{s,a,q}  + \mathds{1}_{s_0}(q) \geq \agentthr \label{cons:reach1}
	\end{align}
\end{subequations}
	where $\policy^{\svisor}_{s,q}$ is the transition probability from $s$ to $q$ under $\policy^{\svisor}$ and the decision variables are $x^{\agent}_{s,a}$ for all $s \in \differset$ and $a \in \actions(s)$. The objective function \eqref{eq:mininfrestime} is obtained by reformulating the KL divergence between the path distributions as the sum of the KL divergences between the successor state distributions for every time step (See Lemma \ref{lemma:globalislocal} in Appendix \ref{section:proofs}). The constraint \eqref{cons:floweqn} encodes the feasible policies and the constraint \eqref{cons:reach1} represents the task constraint.

	\begin{proposition}\label{proposition:optimalisachieved}
		The optimization problem given in \eqref{prog:stateprogram} is a convex optimization problem that shares the same optimal value with \eqref{problem:mindeviation}. Furthermore, there exists a policy $\policy \in \Pi^{St}(\mdp)$ that attains the optimal value of \eqref{prog:stateprogram}.
	\end{proposition}

	The optimization problem given in \eqref{prog:stateprogram} gives the optimal  state-action occupation measures for the agent. One can synthesize the optimal deceptive policy $\policy^{\agent}$ using the relationship $x^{\agent}_{s,a} = \policy^{\agent}_{s,a} \sum_{a' \in A(s)} x^{\policy}_{s,a'}$ for all $s \in \differset$ and $\policy^{\agent}_{s,a} = \policy^{\svisor}_{s,a}$ for the other states.

	The optimization problem given in \eqref{prog:stateprogram} can be considered as a constrained MDP problem with an infinite action space~\cite{altman1999constrained} and a nonlinear cost function. This equivalence follows from that there exists a deterministic policy that incurs the same cost on the infinite action MDP for every randomized policy for $\mdp$. Since there exists a deterministic optimal policy for the infinite MDP, we can represent the objective function and constraints of Problem \ref{problem:mindeviation} with the occupancy measures. However, we remark that \eqref{prog:stateprogram} is a convex nonlinear optimization problem whereas the constrained MDPs are often modeled with a linear cost function and solved using linear optimization methods.

\begin{remark}
		The methods provided in this section can be generalized to task constraints that are co-safe LTL specifications. In detail, every co-safe LTL can be translated into a DFA~\cite{kupferman2006construction}. By combining the MDP and the DFA, we get the product MDP. Since the co-safe LTL specifications translates into reachability specifications on the product MDP and there is a one-to-one mapping between the paths of the original MDP and the product MDP, we can apply the methods described in this section to compute an optimal deceptive policy.
\end{remark}

\section{Synthesis of Optimal Reference Policies} \label{section:reference}
	In this section, we prove the hardness of Problem \ref{problem:mindeviation2}. We give an optimization problem based on dualization approach to synthesize locally optimal reference policies. We provide a distributed optimization algorithm based on ADMM for synthesis of locally optimal reference policies. We also derive a lower bound on the objective function and give a linear programming relaxation of Problem \ref{problem:mindeviation2}.
	
	The optimization problem given in \eqref{prog:stateprogram} has the supervisor's policy parameters as constants. We want to solve the optimization problem given in \eqref{prog:stateprogram} to formulate the synthesis of optimal reference policies by adding the supervisor's policy parameters as additional decision variables. The set $\closedset$ is the set of states that belong to a closed communicating class of $\mdp^{\svisor}$. In  \eqref{prog:stateprogram}, $\closedset$ is a constant set for a given reference policy, but it may vary under different reference policies. We make the following assumption to prevent set $\closedset$ from varying under different reference policies.
	\begin{assumption} \label{assumption:copisfixed}
		The set $\closedset$ is the same for all reference policies considered in Problem \ref{problem:mindeviation2}.
	\end{assumption}
	\begin{remark}
		Assumption \ref{assumption:copisfixed} is made for the clarity of representation. In the absence of Assumption \ref{assumption:copisfixed}, one can to compute the optimal reference policy for different values of $\closedset$. However, we remark that since, in general, $\closedset$ can have $\mathcal{O}(2^{|\states|})$ values, computing the optimal reference policy for different values of $\closedset$ may have exponential complexity in $|\states|$.
	\end{remark}
	Under Assumptions \ref{assumption:supervisorisstationary} and \ref{assumption:copisfixed}, the optimal value of Problem \ref{problem:mindeviation2} is equal to the optimal value of the following optimization problem:
	\begin{subequations}
		\label{prog:stateprogramsupervisor}
		\begin{align}
		&\underset{x^{\svisor}_{s,a}}{\sup} \quad
		\underset{x^{\agent}_{s,a}}{\inf}  \quad  \sum_{s \in \states_{d} } \sum_{a \in \actions(s)} \sum_{q \in \successor(s)} x^{\agent}_{s,a} \probs_{s,a,q} \\
		&\hspace{2cm} \log \left( \frac{\sum_{a' \in \actions(s)} x^{\agent}_{s,a'} \probs_{s,a',q}}{ \policy^{\svisor}_{s,q} \sum_{a' \in \actions(s)} x^{\agent}_{s,a'}} \right) \label{eq:combinedobjective}
		\\
		&\normalfont \text{subject to} \quad \nonumber
		\\
		& \eqref{cons:positiveactions} - \eqref{cons:reach1} \nonumber
		\\
		& \policy^{\svisor}_{s,q} = \sum_{a \in \actions(s)} \probs_{s,a,q} \frac{x^{\svisor}_{s,a}}{\sum_{a' \in \actions(s)} x^{\svisor}_{s,a'}}, \hspace{0.1cm}\forall s \in \states_{d}, \ \forall q \in \states, \label{cons:policyrestimesupervisor} 
		\\
		&  x^{\svisor}_{s,a} \geq 0, \hspace{3.4cm} \forall s \in \states_d, \ 
		\forall a \in \actions(s), \label{cons:positiveactionssupervisor}
		\\
		& \hspace{-0.1cm} \sum_{a \in \actions(s)} x^{\svisor}_{s,a} - \sum_{q \in \states_{d}}  \sum_{a \in \actions(q)} x^{\svisor}_{q,a}\probs_{q,a,s} = \mathds{1}_{s_0}(s), \forall s \in \states_{d},  \label{cons:floweqnsupervisor}
		\\
		& \sum_{q \in R^{\svisor}_{i}} \sum_{s \in \states_{d} \setminus C_{\svisor}}  \sum_{a \in \actions(s)}  x^{\svisor}_{s,a}\probs_{s,a,q}  + \mathds{1}_{s_0}(q) \geq \svisorthri, \  \forall i \in[\svisornumofspec]  \label{cons:reachabilitysupervisor}
		\end{align}
	\end{subequations}
	where $x^{\svisor}_{s,a}$ variables are the decision variables for the supervisor and $x^{\agent}_{s,a}$ variables are the decision variables for the agent. 
	
	\begin{remark}
		The optimization problem given in \eqref{prog:stateprogramsupervisor} has undefined points due to the denominators in \eqref{eq:combinedobjective} and \eqref{cons:policyrestimesupervisor}, that are ignored in the above optimization problem for the clarity of representation. If $\sum_{a \in A(s)}x^{\svisor}_{s,a} = 0$, then the state $s$ is unreachable and if the KL divergence between the policies is finite, the state must be unreachable also under $\policy^{\agent}$. Hence there is no divergence at state $s$. If $\policy^{\svisor}_{s,q} =0$ and if the KL divergence between the policies is finite, $x^{\agent}_{s,q}$ must be $0$. Hence there is no divergence for state $s$ and successor state $q$. 
	\end{remark}
  
	We can show the existence of an optimal reference policy if the condition given in Proposition \ref{proposition:supervisorattains} is satisfied. This condition ensures that the objective function of the problem in  \eqref{prog:stateprogramsupervisor} is finite for all pairs of the supervisor's and the agent's policies.  
	\begin{proposition} \label{proposition:supervisorattains}
		If $\probs_{s,a,q} > 0$ for all $s \in \states_{d}$, $a \in \actions(s)$, and $q \in Succ(s)$, then there exists a policy $\policy^{\svisor}$ that attains the optimal value of the optimization problem given in \eqref{prog:stateprogramsupervisor}.
	\end{proposition}

We note that the optimization problem given in \eqref{prog:stateprogramsupervisor} is nonconvex. One might wonder whether there exists a problem formulation that yields a convex optimization problem. We show that it is not possible to obtain a convex reformulation of the optimization problem given in \eqref{prog:stateprogramsupervisor}.

We first observe that it is possible that there are multiple locally optimal reference policies. For example, consider the MDP given in Figure \ref{fig:minisnotoptimalexample} where the specification of the agent is $\Pr^{\policy^{\agent}}_{\mdp}(s \models \lozenge q_1 \vee \lozenge q_2) = 1$. Regardless of the reference policy, the agent's policy must have $\policy^{\agent}_{s, \gamma} = 1$ due to his specification. For simplicity, there is no specification for the supervisor, i.e., $\nu^{\svisor}$ is $0$. The optimal reference policy maximizes $0.4\log(0.4/ (0.32x^{\svisor}_{s_0,\alpha} + 0.15x^{\svisor}_{s_0,\beta}+ 0.4x^{\svisor}_{s_0,\gamma})) + 0.6\log(0.6/ (0.08x^{\svisor}_{s_0,\alpha} + 0.15x^{\svisor}_{s_0,\beta}+ 0.6x^{\svisor}_{s_0,\gamma}))$, which is a convex function of $x^{\svisor}_{s_0, \alpha}$, $x^{\svisor}_{s_0, \beta}$, and $x^{\svisor}_{s_0, \gamma}$. There are two locally optimal reference policies for Problem \ref{problem:mindeviation2}: the policy that satisfies $\policy^{\svisor}_{s,\alpha} = 1$ and the policy that satisfies $\policy^{\svisor}_{s,\beta} = 1$. Hence, the problem is not only nonconvex but also possibly multimodal. 

\begin{figure}[t] 
	\centering
	\subfloat[]{\scalebox{0.24}{
		\resizebox{1\textwidth}{!}{ \centering
			\begin{tikzpicture} 
			\node[state, initial]  (s0) {$s$};
			\node[state] [above right=of s0] (s1) {$q_1$};
			\node[state] [below right=of s0] (s3) {$q_3$};
			\node[state] [below right=of s1] (s2) {$q_2$};
			\draw 
			(s0) edge[bend left=80] node[above, sloped] {$\alpha,0.32$} (s1)
			(s0) edge[bend right=10] node[above, sloped] {$\beta,0.15$} (s1)
			(s0) edge[bend left=30] node[above, sloped] {$\gamma,0.4$} (s1)
			(s0) edge[bend left] node[below, sloped] {$\alpha,0.08$} (s2)
			(s0) edge node[below, sloped] {$\gamma,0.6$} (s2)
			(s0) edge[bend right] node[below, sloped] {$\beta,0.15$} (s2)
			(s0) edge[bend right=20] node[below, sloped] {$\alpha,0.6$} (s3)
			(s0) edge[bend right=60] node[below, sloped] {$\beta,0.7$} (s3)
			(s1) edge[loop right] node[right] {$\alpha,1$} (s1)
			(s2) edge[loop right] node[right] {$\alpha,1$} (s2)
			(s3) edge[loop right] node[right] {$\alpha,1$} (s3);
			\end{tikzpicture}}
		\label{fig:minisnotoptimalexample}
	}}
	\subfloat[]{\scalebox{0.24}{
		\centering
		\includegraphics[width=\textwidth]{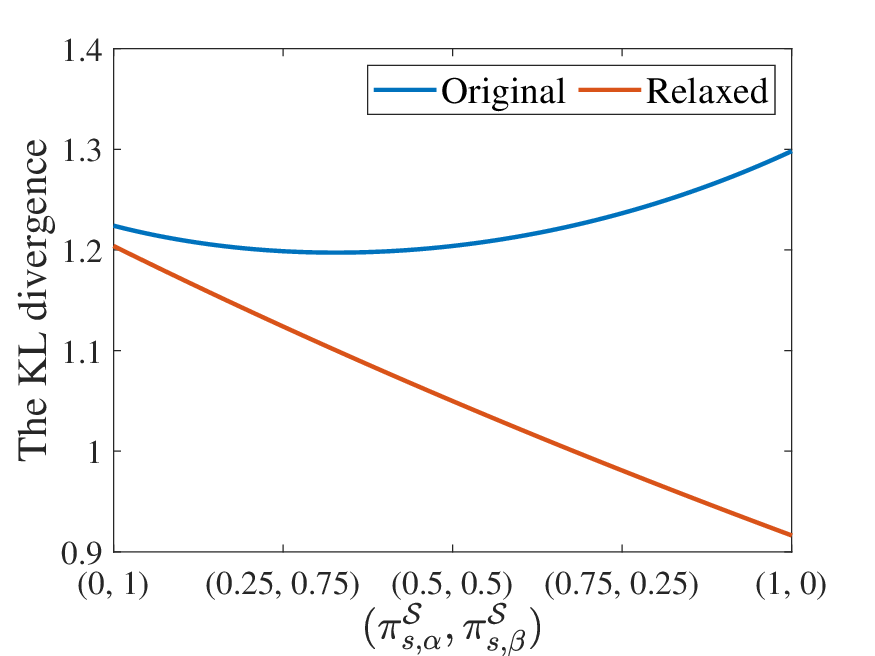}
		\label{fig:minisnotoptmalgraph}
	}}
	\caption{(a) An MDP with 4 states. A label $a,p$ of a transition refers to the transition that happens with probability $p$ when action $a$ is taken. (b) The KL divergence between the path distributions of the agent and the supervisor for different reference policies. Note that there are two local optima that maximizes the KL divergence.}
\end{figure}

We consider a new parametrization to reformulate the optimization problem given in \eqref{prog:stateprogramsupervisor}. Consider a continuous and bijective transformation from the occupation measures to the new parameters, that makes new parameters to span all stationary policies. After this transformation, an optimal solution to \eqref{prog:stateprogramsupervisor} yields an optimal solution in the new parameter space. If the optimization problem given in \eqref{prog:stateprogramsupervisor} has multiple local optima, then any reformulation spanning all stationary policies for the supervisor has multiple local optima. Therefore, it is not possible to obtain a convex reformulation.

		In Section \ref{subsection:hardness}, we show that Problem \ref{problem:mindeviation2} is a provably hard problem. In Section \ref{subsection:dual}, we describe dualization-based procedure to locally solve the optimization problem given in \eqref{prog:stateprogramsupervisor}. As an alternative to solving the dual problem, we give an algorithm based on alternating direction method of multipliers (ADMM) in Section \ref{subsection:ADMM}. Finally, we present a relaxation of the problem in Section \ref{subsection:relaxation} that relies on solving a linear program.

	\subsection{The Complexity of the Synthesis of Optimal Reference Policies}
	\label{subsection:hardness}
	In this section, we show that the synthesis of an optimal reference policy is NP-hard whereas the feasibility problem for reference policies can be solved in polynomial time.

	Finding a feasible policy under multiple reachability constraints has polynomial complexity in the number states and actions for a given MDP. When the target states are absorbing, the complexity of the problem is also polynomial in the number of constraints~\cite{etessami2007multi}. This result follows from that a feasible policy can be synthesized with a linear program where the numbers of variables and constraints are polynomial in the number of states, actions, and task constraints.
	
		 Matsui transformed the set partition problem to the decision version of an instance of linear multiplicative programming and proved the NP-hardness of linear multiplicative programming~\cite{matsui1996np}. In the proof of Proposition \ref{proposition:hardness}, we give an instance of Problem \ref{problem:mindeviation2} whose decision problem reduces into the decision problem of the instance of linear multiplicative programming that Matsui provided\protect{\footnote{The complete proof is available at the end of this document.}}.

	While a feasible reference policy can be synthesized in polynomial time, the complexity of finding an optimal reference policy is NP-hard even when the target states are absorbing. The hardness proof follows from a reduction of Problem \ref{problem:mindeviation2} to an instance of linear multiplicative programming which minimizes the multiplication of two variables subject to linear inequality constraints. Formally we have the following result.
	\begin{proposition} \label{proposition:hardness}
	Problem \ref{problem:mindeviation2} is NP-hard even under Assumption \ref{assumption:closednessanddisjointness}.
	\end{proposition}
	\begin{proof}[\textbf{Sketch of Proof for Proposition \ref{proposition:hardness}}]
		Problem \ref{problem:mindeviation2} can be reduced to linear multiplicative programming.  Linear multiplicative program can be reduced to the set partition problem. Since the set partition problem is NP-hard, Problem \ref{problem:mindeviation2} is NP-hard.
		
		In more detail, the set partition problem~\cite{garey1979computers,karp1972reducibility}  is NP-hard and is the following:
		
		\textbf{Instance: } An $m \times n$ $0-1$ matrix $M$ satisfying $n > m$.
		
		\textbf{Question: } Is there a $0-1$ vector $x$ satisfying 	$	 \sum_{\substack{j=1 \\ M_{ij} = 1}}^{n} x_{j} = 1$ for all $i \in [n]$.

		Linear multiplicative programming minimizes the product of two variables subject to linear inequality constraints and is NP-hard~\cite{matsui1996np} . Let $M$ be an $m \times n$ $0-1$ matrix with $n \geq m$ and $n \geq 5$, and $p = n^{n^{4}}$. The problem 
		
		\small
		\begin{subequations}
			
			\label{matsuiprogram}
			\begin{align}
			&\min \quad && (2p^{4n} - p + 2p^{2n} x_{0} + y_{0}) (2p^{4n} - p - 2p^{2n} x_{0} + y_{0}) \nonumber
			\\
			&\normalfont \text{subject to} && x_{0} = \sum_{i=1}^{n} p^{i} x_{i} \label{matsuiconsfirst}
			\\
			& && y_{0} = \sum_{i=1}^{n} \sum_{j=1}^{n} p^{i+j} y_{ij}
			\\
			& \forall i \in [n], && 0 \leq x_{i} \leq 1, y_{ii} = x_{i},
			\\
			& \forall i ,j \in [n] && 0 \leq y_{ij} \leq 1,  
			\\
			&  \substack{\forall i,j \in [n] \\ i\neq j}, && x_{i}  \geq y_{ij}, \quad x_{j}  \geq y_{ij} , \quad y_{ij}  \geq x_{i} + x_{j} - 1, 
			\\
			&\forall i \in [m], && \sum_{\substack{j=1 \\ M_{ij} = 1}}^{n} x_{j} = 1, \label{consM} 
			\end{align}
			\end{subequations}\normalsize
		where the decision variables are $x_{i}$ for all $i \in [n]$ and $y_{ij}$ for all $i,j \in [n]$, is NP-hard. In detail, \cite{matsui1996np} proved that the optimal value of \eqref{matsuiprogram} is less than or equal to $4p^{8n}$ if and only if there exists a $0-1$ solution for $x_{1}, \ldots, x_{n}$ satisfying \eqref{consM}. Since the decision problem of \eqref{matsuiprogram} correspond to solving the set partition problem, \eqref{matsuiprogram} is NP-hard. 
		
		We can construct an MDP with a size polynomial in $n$ and choose polynomial number of specifications in $n$ such that the optimal value of Problem \ref{problem:mindeviation2} is 
		\small
		\begin{subequations} \label{opt:finalverison}
			\begin{align*}
			&\max \frac{1}{2} \log\frac{1}{(2p^{4n} - p + 2p^{2n} x_{0} + y_{0}) (2p^{4n} - p - 2p^{2n} x_{0} + y_{0})} \nonumber
			\\
			& \quad \quad+ \frac{1}{2} \log\left( 4C^2(n^2 + n + 1)^2\right) 
			\\
			&\normalfont \text{subject to } \quad \eqref{matsuiconsfirst}-\eqref{consM}
			\end{align*}
		\end{subequations}
		\normalsize
		where $C$ is a constant depending on $n$. Due to the result given in \cite{matsui1996np}, the optimal value of \eqref{opt:finalverison} is greater than or equal to $-\log(4p^{8n})/2 + \log\left( 4C^2(n^2 + n + 1)^2\right)/2 $ if and only if there exists a $0-1$ solution for $x_{1}, \ldots, x_{n}$ satisfying \eqref{consM}. Since the decision problem of \eqref{opt:finalverison} correspond to solving the set partition problem, \eqref{opt:finalverison} is NP-hard.  
		
		Since the number of states, actions, and the task constraints is polynomial in $n$ and \eqref{opt:finalverison} synthesizes an optimal reference policy, the synthesis of optimal reference policies is NP-hard.

	\end{proof}

	We remark that the hardness of Problem \ref{problem:mindeviation2} is due to the nonconvexity of the KL objective function since the feasibility problem can be solved in polynomial time.

	\subsection{Dualization-based Approach for the Synthesis of Optimal Reference Policies} \label{subsection:dual}
	Observing that Slater's condition~\cite{boyd2004convex} is satisfied, and the strong duality holds for the optimization problem given in \eqref{prog:stateprogram}, to find the optimal value of \eqref{prog:stateprogramsupervisor} one may consider solving the dual of \eqref{prog:stateprogram} with $x^{\svisor}_{s,a}$ as additional decision variables and \eqref{cons:policyrestimesupervisor}-\eqref{cons:reachabilitysupervisor} as additional constraints. In this section, we describe the dualization-based approach for the synthesis of locally optimal reference policies.
	
		The optimization problem given in \eqref{prog:stateprogram} has the following conic optimization representation: 	\begin{subequations}
		\label{prog:conicstateprogram}
		\begin{align}
		\underset{y}{\min} \ &  c^{T} y
		\\
		\normalfont \text{subject to } \
		&  [G | -I] y = h,
		\\
		& y \in \mathcal{K}.
		\end{align}
	\end{subequations}
	
	We construct the parameters of the above optimization problem as follows. Define the variable $r_{(s,q)}$ for all $s \in \states_{d}$ and $q \in \successor(s)$. Let $r$ be the $M \times 1$ vector of $r_{(s,q)}$ variables where $r_{(s,q)}$ has the index $(s,q)$. The conic optimization problem has the objective function $\sum_{s \in \states_{d}} \sum_{q \in  \successor(s)} r_{(s,q)}$ and the constraint \begin{equation} \label{cons:exp} r_{(s,q)} \geq \sum_{a \in \actions(s)}  x^{\agent}_{s,a} \probs_{s,a,q} \log \left( \frac{\sum_{a' \in \actions(s)} x^{\agent}_{s,a'} \probs_{s,a',q}}{ \policy^{\svisor}_{s,q} \sum_{a' \in \actions(s)} x^{\agent}_{s,a'}} \right) \end{equation} for all $s \in \states_{d}$ and $q \in  \successor(s)$. The $N \times 1$ vector of $x^{\agent}_{s,a}$ variables is $x^{\agent}$ where $x^{\agent}_{s,a}$ has index $s,a$. Define $y = [x^{\agent},r]^T$. We encode constraint \eqref{cons:floweqn} with $G_{eq}y = h_{eq}$ where $G_{eq}$ is a $N \times (N+M)$ matrix with $(s, (q,a))$-th entry $\mathds{1}_{s}(q) - \probs_{q,a,s}$, and $s$-th entry of $h$ is $\mathds{1}_{s_0}(s).$ The constraint \eqref{cons:positiveactions} is encoded by $G_{+}y\geq 0$ where $G_{+}: = [I_{N \times N} | 0_{N \times M} ]$. The additional constraint given in \eqref{cons:exp} is encoded by $G_{(s,q)} y \in K_{\exp}$ where $G_{(s,q)}$ is a $3 \times (N+M)$ matrix with $(1,N+(s,q))$-th entry $-1$, $(2,(s,a))$-th entry $\probs_{s,a,q}$ for all $a \in \actions(s)$,  $(3,(s,a))$-th entry $\policy^{\svisor}_{s,q}$ for all $a \in \actions(s)$. The constraint \eqref{cons:reach1} is encoded by $G_{\agent} y \geq \nu^{\agent}$ where $G_{\agent}$ is a $1 \times (N+M)$ matrix where $(1,(s,a))$-th entry is $\mathds{1}_{\states_{d} \setminus \agentsset}(s) \sum_{q \in \agentsset} \probs_{s,a,q}$. Finally, $\mathcal{K} = \mathbb{R}^{N + M} \times \lbrace 0 \rbrace^{|\states_{d}|} \times \mathbb{R}^{N}_{+} \times \mathcal{K}_{\exp} \times \ldots \times \mathcal{K}_{\exp} \times \mathbb{R}_{+} $, $G = [ G_{eq},G_{+},G_{(1,1)}, \ldots, 	G_{(|\states_{d}|,|\states|)},G_{\agent} ]^{T}$, $h = [	h_{eq} , 0 ,	\ldots,	0 ,\nu^{\agent}]$, and $ c = [ 0_{N \times 1},	1_{M \times 1}]$.

	The dual of the optimization problem in \eqref{prog:conicstateprogram} is 
	\begin{subequations}
		\label{prog:dualconicprogram}
	\begin{align}
	\underset{u,w}{\max} \ &  h^{T} u \label{eq:dualobj}
	\\
	\normalfont \text{subject to} \ & \begin{bmatrix}
	G^T\\ 
	-I^T
	\end{bmatrix} u + w = c,
	\\
	&w \in \mathcal{K}^{*},
	\end{align}
\end{subequations}
where the decision variables are $u$ and $w$, and $\mathcal{K}^{*} = \lbrace 0\rbrace ^{N+M} \times \mathbb{R}^{|\states_{d}|} \times \mathbb{R}^{N}_{+} \times \mathcal{K}^{*}_{\exp} \times \ldots \times \mathcal{K}^{*}_{\exp} \times \mathbb{R}_{+}$. 

By combining the optimization problem in \eqref{prog:dualconicprogram} and the constraints in \eqref{cons:policyrestimesupervisor}-\eqref{cons:reachabilitysupervisor}, and adding $x^{\svisor}_{s,a}$ as decision variables, we get an optimization problem that shares the same optimal value with \eqref{prog:stateprogramsupervisor}. However, we remark that this problem is nonconvex because of the constraint \eqref{cons:policyrestimesupervisor} and the bilinear constraints that are due to $\pi^{\svisor}_{s,q}$ parameter introduced in the construction of $G_{(s,q)}$. 

\subsection{Alternating Direction Method of Multipliers (ADMM)-based Approach for the Synthesis of Optimal Reference Policies} \label{subsection:ADMM}
The alternating direction method of multipliers (ADMM)~\cite{boyd2011distributed} is an algorithm to solve decomposable optimization problems by solving smaller pieces of the problem. We use the ADMM to locally solve the optimization problem given in \eqref{prog:stateprogramsupervisor}. The objective function of \eqref{prog:stateprogramsupervisor} is decomposable since it is a sum across $\states_{d}$ where each summand consists of different variables. We exploit this feature to reduce the problem size via the ADMM. 

For every state $s \in S_{d}$, we introduce $z^{\agent}_{s}$ and $z^{\svisor}_{s}$ such that $z^{\agent}_{s} = x^{\agent}_{s}$ and $z^{\svisor}_{s} = x^{\svisor}_{s}$. With these extra variables, the augmented Lagrangian of \eqref{prog:stateprogramsupervisor} is 

	\begin{align*}
	&L(x^{\svisor}, x^{\agent}, z^{\svisor}, z^{\agent}, \lambda^{\svisor}, \lambda^{\agent}) 
	\\
	&= \sum_{s \in \states_{d} } \sum_{a \in \actions(s)} \sum_{q \in \states} \left( 
	x^{\agent}_{s,a} \probs_{s,a,q} \log \left( \frac{\sum_{a' \in \actions(s)} x^{\agent}_{s,a'} \probs_{s,a',q}}{ \policy^{\svisor}_{s,q} \sum_{a' \in \actions(s)} x^{\agent}_{s,a'}} \right) \right. \\
	&- \mathcal{I}_{\mathbb{R}^{|\actions(s)|}_{\geq 0}} (x^{\svisor}_{s}) + \mathcal{I}_{\mathbb{R}^{|\actions(s)|}_{\geq 0}} (x^{\agent}_{s})  - \rho^{\svisor}(x^{\svisor}_{s} - z^{\svisor}_{s})^{T}\lambda^{\svisor}_{s} 
	\\
	& + \rho^{\agent}(x^{\agent}_{s} - z^{\agent}_{s})^{T}\lambda^{\agent}_{s}  -  \frac{\rho^{\svisor}}{2} \| x^{\svisor}_{s} - z^{\svisor}_{s} \|^{2}_{2} \left.   +\frac{\rho^{\agent}}{2} \|  x^{\agent}_{s} - z^{\agent}_{s} \|^{2}_{2} \right)   
	\\
	&- \mathcal{I}_{X^{\svisor}}(z^{\svisor}) + \mathcal{I}_{X^{\agent}}(z^{\agent}),
	\end{align*}

where $\rho^{\svisor}$ and $\rho^{\agent}$ are positive constants, $\lambda^{\svisor}$ and $\lambda^{\agent}$ are the dual parameters, $X^{\agent}$ is the set of occupation measures of the agent that satisfy \eqref{cons:floweqn} and \eqref{cons:reach1}, $X^{\svisor}$ is the set of occupation measures of the supervisor that satisfy \eqref{cons:floweqnsupervisor} and \eqref{cons:reachabilitysupervisor}, and  $\policy^{\svisor}_{s,q} = \sum_{a \in \actions(s)} \probs_{s,a,q} {x^{\agent}_{s,a}}/({\sum_{a' \in \actions(s)} x^{\agent}_{s,a'}})$ for all $s \in \states_{d}$ and $a \in \actions(s)$. In Algorithm \ref{algo:admm} which is a modified version of the classical ADMM, we give the ADMM for the synthesis of reference policies. Note that we optimize $x^{\svisor}$ and $x^\agent$ together to capture the characteristics of the maximin problem. 

	\begin{algorithm} 
	\caption{The ADMM for the synthesis of reference policies} \label{algo:admm}
	\begin{algorithmic}[1]
		\State \textbf{Input:} An MDP $\mdp$, reachability specifications $\lozenge R^{\svisor}_{i}$ for all $i \in \svisornumofspec$ and $\agentspec$, probability thresholds $\svisorthri$ for all $i \in [\svisornumofspec]$ and $\agentthr$.
		\State \textbf{Output:} A reference policy $\policy^{\svisor}$.
		\State Set $x^{\svisor, 0}$ and $z^{\svisor, 0}$ arbitrarily from $X^{\svisor}$. 
		\State Set $x^{\agent, 0}$ and $z^{\agent, 0}$ arbitrarily from $X^{\agent}$. 
		\State Set $\lambda^{\svisor, 0}$ and $\lambda^{\agent, 0}$ to 0. 
		\State $k = 0$.
		\While{stopping criteria are not satisfied}
		\State Set $x^{\svisor, k+1}$ and $x^{\agent, k+1}$ as the solution of $\max_{x^{\svisor}} \min_{x^{\agent}} L(x^{\svisor}, x^{\agent}, z^{\svisor, k}, z^{\agent, k}, \lambda^{\svisor , k}, \lambda^{\svisor, k})$. \label{step:optimizemaxmin}
		\State $z^{\svisor, {k+1}} := Proj_{X^{\svisor}}(x^{\svisor, k+1} + \lambda^{\svisor, k})$.
		\State $z^{\agent, k+1} := Proj_{X^{\agent}}(x^{\agent, k+1} + \lambda^{\agent, k})$.
		\State $\lambda^{\svisor, k+1} := \lambda^{\svisor, k} + x^{\svisor, k+1} - z^{\svisor, k+1}$.
		\State $\lambda^{\agent, k+1} := \lambda^{\agent, k} + x^{\agent, k+1} - z^{\agent, k+1}$.
		\State $k:=k+1$.
		\EndWhile 
		\State Compute $\policy^{\svisor}$ using $z^{\svisor, {k}}$ as the occupation measures.
	\end{algorithmic}
\end{algorithm}

We remark that Algorithm \ref{algo:admm} still requires solving a maximin optimization problem (see line \ref{step:optimizemaxmin}). However, the maximin optimization problem in Algorithm \ref{algo:admm} can be solved as a local maximin problem separately for each state since $x^{\svisor}_{s}$ and $x^{\agent}_{s}$ are decoupled from $x^{\svisor}_{q}$ and $x^{\agent}_{q}$ for all $s \neq q \in S_{d}$.  While the number of variables for the problem obtained via dualization-based approach is $\mathcal{O}(|S||A|)$, it is $\mathcal{O}(|A|)$ for the local problems in the ADMM algorithm. 

Since the strong duality holds, one can use a dualization-based approach as shown in Section \ref{subsection:dual} to solve the local maximin problems. We remark that after dualization, the resulting optimization problems are nonconvex similar to the optimization problem obtained via dualization-based approach. 

\begin{remark}
	Convergence of ADMM for particular nonconvex optimization problems has been studied~\cite{wang2015global, hong2016convergence}. To the best of our knowledge, the method based on the ADMM for the optimization problem given in \eqref{prog:stateprogramsupervisor} has no convergence guarantees and does not match with the any of the existing convergence results.
\end{remark}

\subsection{A Linear Programming Relaxation for the Synthesis of Reference Policies} \label{subsection:relaxation}

Since it is not possible to obtain a convex reformulation of the optimization problem given in \eqref{prog:stateprogramsupervisor} via a transformation, we give a convex relaxation of the problem. Intuitively, synthesizing a policy that minimizes the probability of satisfying the agent's specification is a good way to increase the KL divergence between the distributions of paths. Formally, consider a transformation of the path distributions that groups paths of $\mdp$ into two subsets: the paths that satisfy $\agentspec$ and the paths that do not satisfy $\agentspec$. After this transformation, the probability assigned to the first subset is $\Pr^{\policy^{\svisor}}_{\mdp}(\initialstate \models \agentspec)$ under policy $\policy^{\svisor}$ and $\Pr^{\policy^{\agent}}_{\mdp}(\initialstate \models \agentspec)$ under policy $\agentspec$. By the data processing inequality given in \eqref{ineq:dataprocessing}, this transformation yields a lower bound on the KL divergence between the path distributions: $KL\left(\Gamma^{\policy^{\agent}}_{\mdp} || \Gamma^{\policy^{\svisor}}_{\mdp} \right)$ is greater than or equal to 
\begin{equation} \label{pathsvssatisfaction}
\small{KL \left( Ber\left({\Pr}^{\policy^{\agent}}_{\mdp}\left(\initialstate \models \agentspec\right)\right) \big|\big| Ber\left({\Pr}^{\policy^{\svisor}}_{\mdp}\left(\initialstate \models \agentspec\right)\right) \right).}
\end{equation}

	We use this lower bound to construct the relaxed problem \begin{subequations}
		\label{problem:mindeviationrelaxed}
		\begin{align}
		\underset{ \policy^{\svisor} \in \Policy(\mdp)}{\sup} \ \underset{ \policy^{\agent} \in \Policy(\mdp)}{\inf} \quad
		&   \eqref{pathsvssatisfaction}
		\\
		\text{\normalfont subject to} \quad
		& {\Pr }^{\policy^{\agent}}_{\mdp}(\initialstate \models \agentspec) \geq \agentthr, \label{cons:agenttask2}
		\\
		& {\Pr} ^{ \policy^{\svisor}}_{\mdp}( \initialstate \models \svisorspeci) \geq \svisorthri, \ i \in [\svisornumofspec]. \label{cons:supervisorstask2}
		\end{align}
	\end{subequations}
	
	If $\Pr^{\policy^{\svisor}}_{\mdp}(\initialstate \models \agentspec) \geq \agentthr$, the agent may directly use the reference policy. Without loss of generality, assuming that $\Pr^{\policy^{\svisor}}_{\mdp}(\initialstate \models \agentspec) < \agentthr$, the objective function of above optimization problem is decreasing in  $\Pr^{\policy^{\svisor}}_{\mdp}(\initialstate \models \agentspec)$ and increasing in $\Pr^{\policy^{\agent}}_{\mdp}(\initialstate \models \agentspec)$. Hence, the problem  \begin{subequations}
		\label{problem:mindeviationmorerelaxed}
		\begin{align}
		\underset{ \policy^{\svisor} \in \Policy(\mdp)}{\sup} \ \underset{ \policy^{\agent} \in \Policy(\mdp)}{\inf} \quad
		&   {\Pr} ^{ \policy^{\agent}}_{\mdp}( \initialstate \models \agentspec) - {\Pr} ^{ \policy^{\svisor}}_{\mdp}( \initialstate \models \agentspec) \label{eq:pathprobmininf2}
		\\
		\text{\normalfont subject to} \quad
& {\Pr }^{\policy^{\agent}}_{\mdp}(\initialstate \models \agentspec) \geq \agentthr, \label{cons:agenttask2}
\\
& {\Pr} ^{ \policy^{\svisor}}_{\mdp}( \initialstate \models \svisorspeci) \geq \svisorthri, \ i \in [\svisornumofspec]. \label{cons:supervisorstask2}
		\end{align}
	\end{subequations}
	shares the same optimal policies with the problem given in \eqref{problem:mindeviationrelaxed}. We note that the optimization problem given in \eqref{problem:mindeviationmorerelaxed} can be solved separately for the supervisor's and the agent's parameters where both of the problems are linear optimization problems. The optimal reference policy for the relaxed problem is the policy that minimizes $\Pr^{\policy^{\svisor}}_{\mdp}(s_0 \models \agentspec)$ subject to $\Pr^{\policy^{\svisor}}_{\mdp}(s_0 \models \svisorspeci) \geq \svisorthri$ for all $i \in [\svisornumofspec]$.
	
	The lower bound given in \eqref{pathsvssatisfaction} provides a sufficient condition on the optimality of a reference policy for Problem \ref{problem:mindeviation2}. A policy $\policy^{\svisor}$ satisfying $\Pr^{\policy^{\svisor}}_{\mdp}(s_0 \models \agentspec) = 0$ and $\Pr^{\policy^{\svisor}}_{\mdp}(s_0 \models \svisorspeci) \geq \svisorthri$ for all $i \in [\svisornumofspec]$ is an optimal reference policy since the optimization problem given in \eqref{problem:mindeviationrelaxed} has the optimal value of $\infty$. However, in general the gap due to the relaxation may get arbitrarily large, and the reference policy synthesized via \eqref{problem:mindeviationrelaxed} is not necessarily optimal for Problem \ref{problem:mindeviation2}. For example, consider  the MDP given in Figure \ref{fig:minisnotoptimalexample} where the agent's policy again has $\policy^{\agent}_{s, \gamma} = 1$. For simplicity, there is no specification for the supervisor, i.e., $\nu^{\svisor}$ is $0$. The policy $\policy^{\svisor}$ that minimizes $\Pr^{\policy^{\svisor}}_{\mdp}(s \models \lozenge q_1 \vee \lozenge q_2)$ chooses action $\beta$ at state $s$. This policy has a KL divergence value of $1.22$. On the other hand, a policy that chooses action $\alpha$ is optimal and it has a KL divergence value of $1.30$ even though it does not minimize the probability of satisfying $\lozenge q_1 \vee \lozenge q_2$. The gap of the lower bound may get arbitrarily large as $\probs_{s,\alpha,q_{2}}$ decreases. Furthermore, the policy synthesized via the relaxed problem may not even be locally optimal as $\probs_{s,\alpha,q_{2}}$ decreases.

The relaxed problem focuses on only one event, achieving the malicious objective, and fails to capture all transitions of the agent. On the other hand, the objective function of Problem \ref{problem:mindeviation2}, the KL divergence between the path distributions, captures all transitions of the agent rather than a single event. In particular, to detect the deviations the optimal deceptive policy assigns a low probability to the transition from $s$ to $q_{2}$ which inevitably happens with high probability for the agent. However, the policy synthesized via the relaxed problem fails to capture that the agent have to assign high probability to the transition from $s$ to $q_{2}$.

\section{Numerical Examples} \label{section:examples}
In this section we give numerical examples on the synthesis of optimal deceptive policies and optimal reference policies. In Section~\ref{subsection:chcofdeceptive} we explain some characteristics of the optimal deceptive policies through different scenarios. In the second example given in Section~\ref{subsection:whyoursisthebest}, we compare the proposed metric, the KL divergence between the distributions of paths, to some other metrics. We demonstrate the ADMM-based algorithm with the example given in Section~\ref{subsection:admmexample}.

We solved the convex optimization problems with CVX \cite{cvx} toolbox using MOSEK \cite{mosek} and the nonconvex optimization problems using IPOPT \cite{wachter2006implementation}.

\subsection{Some Characteristics of Deceptive Policies} \label{subsection:chcofdeceptive}
The first example demonstrates some of the characteristics of the optimal deceptive policies. The environment is a $20 \times 20$ grid world given in Figure \ref{fig:example1}. The green and red states are denoted with sets $g$ and $r$, respectively. At every state, there are $4$ available actions, namely, up, down, left, and right. When the agent takes an action the transition happens into the target direction with probability $0.7$ and in the other directions uniformly randomly with probability $0.3$. If a direction is out of the grid, the transition probability of that direction is proportionally distributed to the other directions. The green and red states are absorbing. The initial state is the top-left state. 
	\begin{figure} [h] 
	\centering
	\subfloat[Reference]{\scalebox{0.2}{
		\centering
		\includegraphics[width=\textwidth]{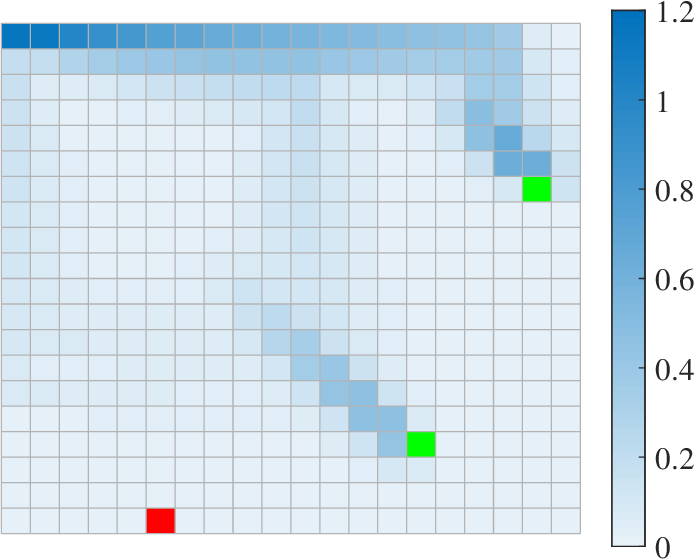}
		\label{fig:example1suppolicy}
	}}
	\subfloat[Deceptive]{\scalebox{0.2}{
		\centering
		\includegraphics[width=\textwidth]{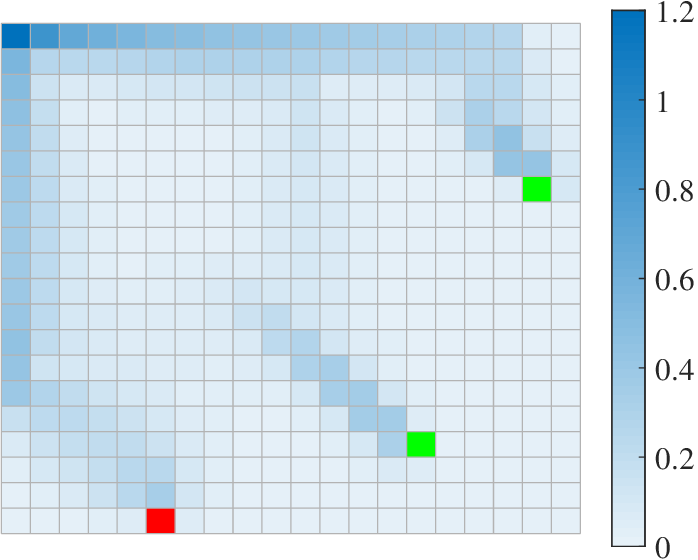}
		\label{fig:example1agentpolicy}
	}}
	\caption{Heat maps of the occupation measures. The value of a state is the expected number of visits to the state. The deceptive policy makes the agent move towards the red state to achieve the malicious objective.}
	\label{fig:example1}
\end{figure}

The specification of the supervisor is to reach each of the green states with probability at least $0.4$. The reference policy is constructed so that it reaches the green states with probability at least $1-10^{-4}$ in the minimum expected time. The specification of the agent is to reach the red state. The specification is encoded with $\lozenge r$. The probability threshold $\agentthr$ for the agent's specification is $0.3$.  We synthesize the policy of the agent according to Problem \ref{problem:mindeviation}, which leads to the KL divergence value of $2.662$. While the reference policy satisfies $\lozenge r$ with probability $10^{-4}$, the agent's policy satisfies $\lozenge r$ with probability $0.3$. 
 
 In Figure \ref{fig:example1agentpolicy}, we observe that if the deceptive agent is close to the green states, it does not try to reach the red state since deviations from the reference policy in these regions incur high divergence. Instead, as we see in Figure \ref{fig:example1policies}, the deceptive policy makes the agent move towards left in the first steps and reach the red state by going down. The misleading occurs during this period: while the agent goes left on purpose, it may hold the stochasticity of the environment accountable for this behavior. We also observe a significant detail in the agent's deceptive policy. The deceptive policy aims to reach the left border since the reference policy takes action down in this region. The agent wants to drive himself to this region to directly follow the reference policy without any divergence. Thus the agent deviates from the reference policy at a particular state to be close to the reference policy as much as possible in the rest of the path. Once the agent is close to the red state, it again deviates from the reference policy and takes action down with a high probability to reach the red state. 
 \begin{figure} [H]
 	\centering
	\subfloat[Supervisor - Left]{\scalebox{0.2}{
 		\centering
 		\includegraphics[width=\textwidth]{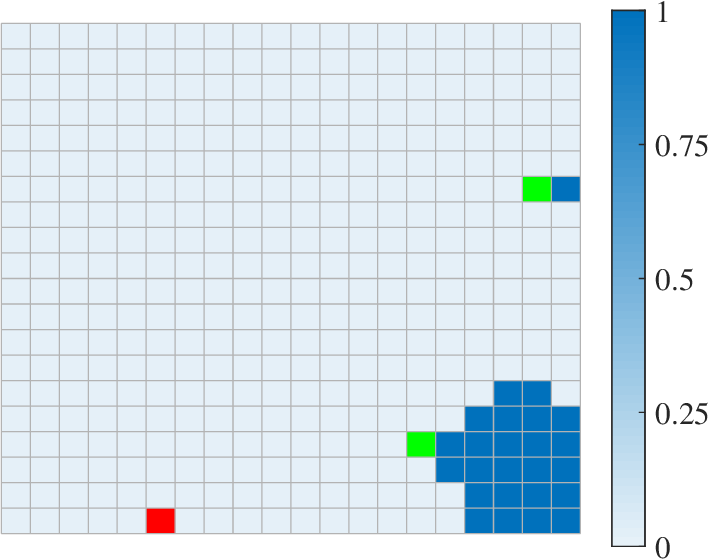}
 		\label{fig:supA1Right}
 	}}
 	\subfloat[Agent - Left]{\scalebox{0.2}{
 		\centering
 		\includegraphics[width=\textwidth]{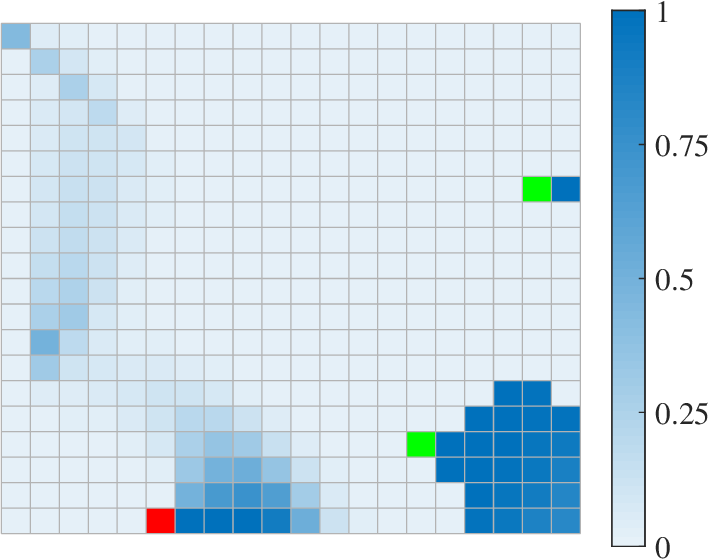}
 		\label{fig:agentA1Right}
 	}}
 
 	\subfloat[Supervisor - Down]{\scalebox{0.2}{
 	\centering
 	\includegraphics[width=\textwidth]{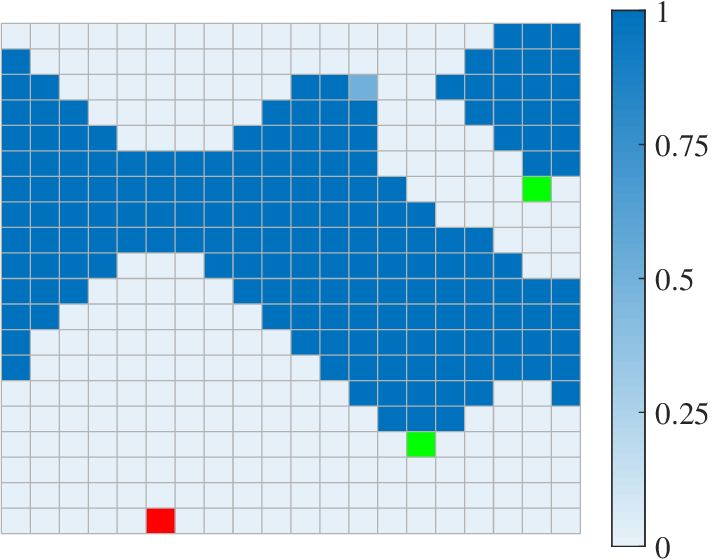}
 	\label{fig:supA1Down}
	}}
 	\subfloat[Agent - Down]{\scalebox{0.2}{
 	\centering
 	\includegraphics[width=\textwidth]{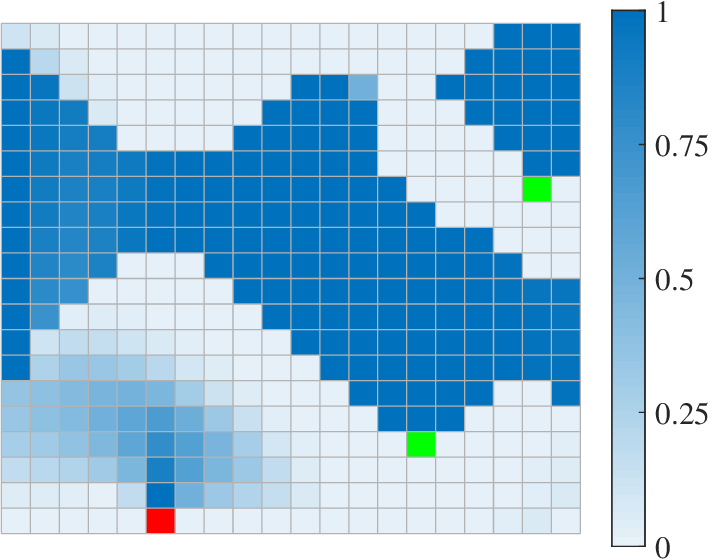}
 	\label{fig:agentA1Down}
	}}
 	\caption{The assigned probabilities to the actions when the yellow state was visited, but the red state was not visited.}
 	\label{fig:example1policies}
 \end{figure}

We note that the reference policy is restrictive in this case; as can be seen in Figure \ref{fig:example1suppolicy}, it follows almost a deterministic path. Under such a reference policy, even the policy that is synthesized via Problem \ref{problem:mindeviation} is easy to detect. To observe the effect of the reference policy on the deceptive policy, we consider a different reference policy as shown in Figure \ref{fig:example1entropysuppolicy}, which satisfies $\lozenge r$ with probability $10^{-3}$. When the reference policy is not as restrictive, the deceptive policy becomes hard to detect. Formally, the value of the KL divergence reduces to $1.462$. 
	
\begin{figure} [h]
	\centering
	\subfloat[Supervisor]{\scalebox{0.2}{
		\centering
		\includegraphics[width=\textwidth]{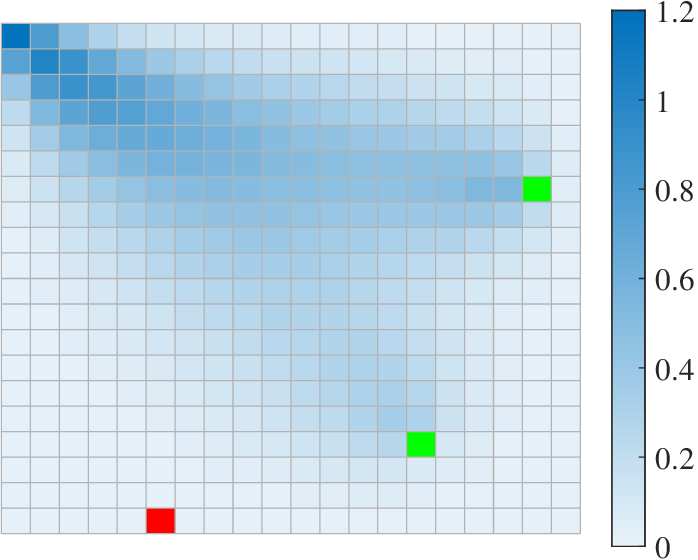}
		\label{fig:example1entropysuppolicy}
	}}
	\subfloat[Agent]{\scalebox{0.2}{
		\centering
		\includegraphics[width=\textwidth]{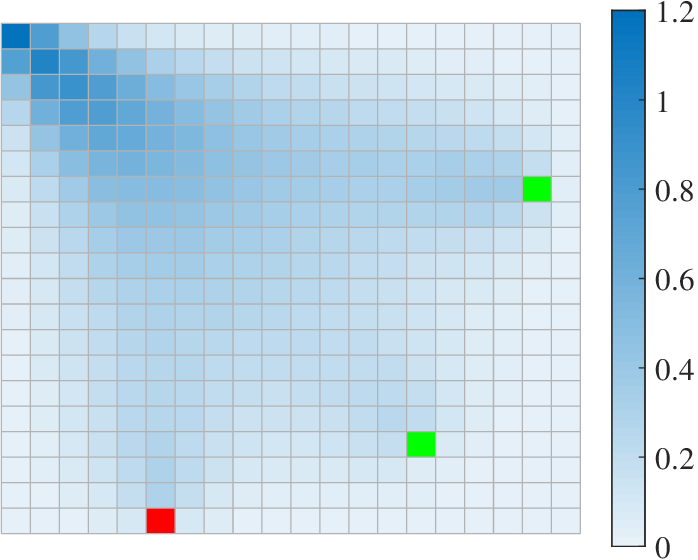}
	}}
	\label{fig:example1entropy}
	\caption{Heatmaps of the occupation measures. The deceptive policy is hard to detect under a reference policy that is not restrictive.}
\end{figure}

\subsection{Detection of a Deceptive Agent} \label{subsection:whyoursisthebest} 
In this example, by comparing KL divergence with some common metrics to synthesize the deceptive policies, we show how the choice of KL divergence helps with preventing detection. We compare the metrics using a randomly generated MDP and an MDP modeling a region from San Francisco.

The randomly generated MDP consists of $21$ states. In particular, there are $20$ transient states with $4$ actions and an absorbing state with $1$ action. For the transient states, each action has a successor state that is chosen uniformly randomly among the transient states. In addition to these actions, every transient state has an action that has the absorbing state as the successor state. At every transient state, the reference policy goes to the absorbing state with probability $0.15$ and the other successor states with probability $0.85$. The agent's specification $\specification^{\agent}$ is to reach to a specific transient state.

We randomly generate a reference policy for the randomly generated MDP. The reference policy satisfies the agent's specification with probability $0.30$. For the reference policy, we synthesize three candidate policies for deception: by minimizing the KL divergence between the path distributions of the agent's policy and the reference policies, by minimizing the $L_{1}$-norm between the occupation measures of the state-action pairs for the agent's policy and the reference policies, and by minimizing the $L_{2}$-norm between the occupation measures of the state-action pairs for the agent's policy and the reference policies. The candidate policies are constructed so that they satisfy the agent's specification $\specification^{\agent}$ with probability $0.9$. For each candidate policy, we run $100$ simulations each of which consists of $100$ independently sampled paths. 

We also simulate the agent's trajectories under the reference policies. In particular, we aim to observe the case where the empirical probability of satisfying $\specification^{\agent}$ is approximately $0.9$. Note that this is a rare event under the reference policy. We simulate this rare event in the following way. Let $\Gamma^{\policy^{\svisor}}_{\mdp}$ be the probability distribution of paths under the reference policy. We create two conditional probability distributions $\Gamma^{\policy^{\svisor}}_{\mdp,+}$ and $\Gamma^{\policy^{\svisor}}_{\mdp,-}$ which are the distribution of paths under the reference policy given that the paths satisfy $\specification^{\agent}$ and do not satisfy $\specification^{\agent}$, respectively. We sample from $\Gamma^{\policy^{\svisor}}_{\mdp,-}$ with probability $0.9$ and $\Gamma^{\policy^{\svisor}}_{\mdp,-}$ with probability $0.1$. 

\begin{figure}[h]
	\centering
	\includegraphics[width=0.2\paperwidth]{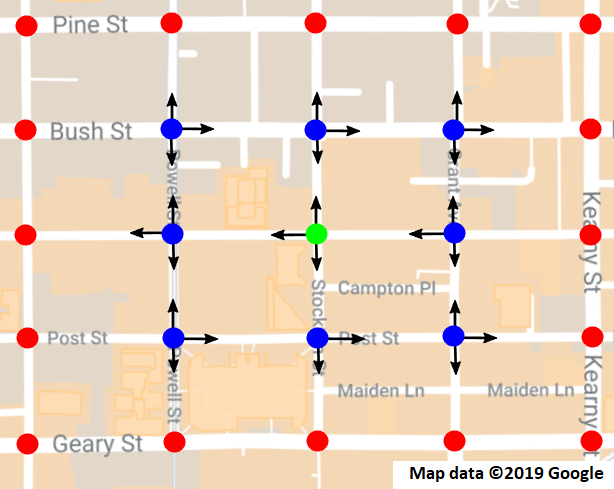}
	\caption{ The map of a region from north east of San Francisco. The green dot indicates the intersection at which the highest number of crimes happened. The data is from \cite{alamdari2014persistent}. The dots on the map represent the states of the MDP and the arrows represent the available actions. The initial state is chosen uniformly randomly among the blue states and the red states are absorbing. The agent aims to patrol the green state.}
	\label{fig:SFmap}
\end{figure}

In addition to the randomly generated MDP, we use a different MDP to show that the deceptive policy can help patrolling without being detected. The MDP models a region in the north east of San Francisco. The map of the region is given in Figure \ref{fig:SFmap} where each intersection is represented with a state and each road is represented with an action. We design the reference policy to represent the average driver behavior. We obtain the traffic density data from Google Maps~\cite{googlemaps} and synthesize the reference policy by fitting a stationary policy to the data. The aim of the agent is to patrol the intersection at which the highest number of crimes happens. Formally, the agent's policy reaches the intersection with probability at least $0.9$ whereas the reference policy reaches the intersection with probability $0.28$. For the simulation, we use the steps as in the randomly generated MDP.

For each simulation, we plot the log-probability under the reference policy and the log-likelihood ratio compared to the reference policy in Figure \ref{fig:loglgraph}. In Figures \ref{fig:loglrandom} and \ref{fig:loglsf}, we observe that the metrics based on $L_1$ and $L_2$ norms generate paths that are more probable under the reference policy compared to KL divergence metric. However, we also observe that the probabilities for the KL divergence metric looks closer to the probabilities of the simulations of the reference policy that are satisfying $\specification^{\agent}$ with high probability. The paths generated via the KL divergence coincides with the typical set of this rare event. In other words, the other metrics provide paths that have higher probabilities under the reference policy. However, given that the empirical distribution of the paths satisfy $\specification^{\agent}$ with high probability, the likelihood of the path generated by the KL divergence is higher and the supervisor expects to see paths that are close to the paths generated by the KL divergence metric. This phenomenon confirms the motivation provided by Sanov's theorem which is explained in Section~\ref{section:problemstatement}.  In Figures \ref{fig:llrtrandom} and \ref{fig:llrtsf}, we observe that the paths generated under the KL divergence metric has a lower log-likelihood ratio compared to the other metrics as explained in Section \ref{section:problemstatement}. This result shows that compared to the other candidate deceptive policies, the deceptive policy generated with the KL divergence is the least likely to be detected under the likelihood-ratio test.

 \begin{figure} [h]
	\centering
	\subfloat[Random MDP]{\scalebox{1}{
		\centering
\includegraphics[width=0.2\paperwidth]{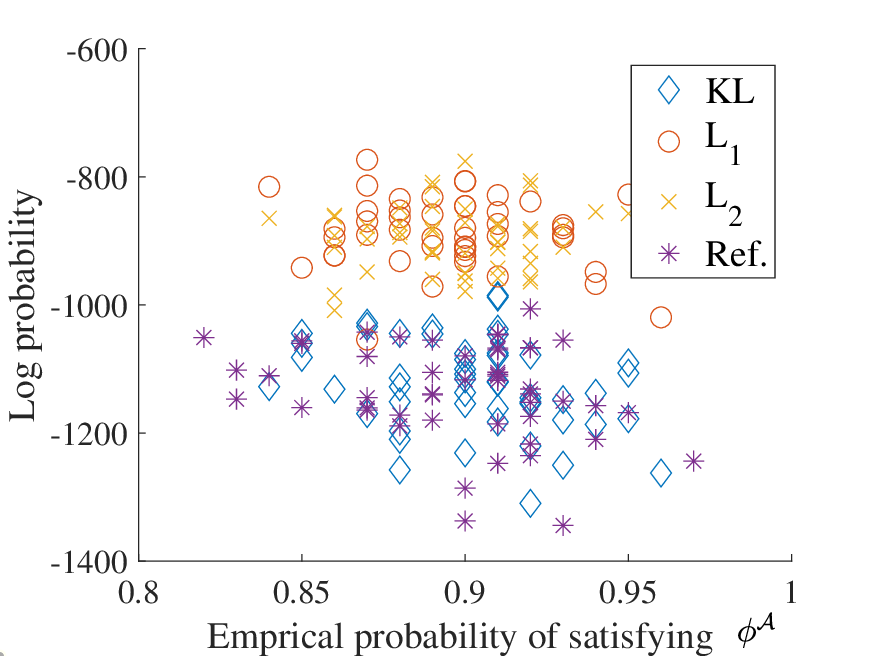}
		\label{fig:loglrandom}
	}}
	\subfloat[MDP for San Francisco]{\scalebox{1}{
	\centering
\includegraphics[width=0.2\paperwidth]{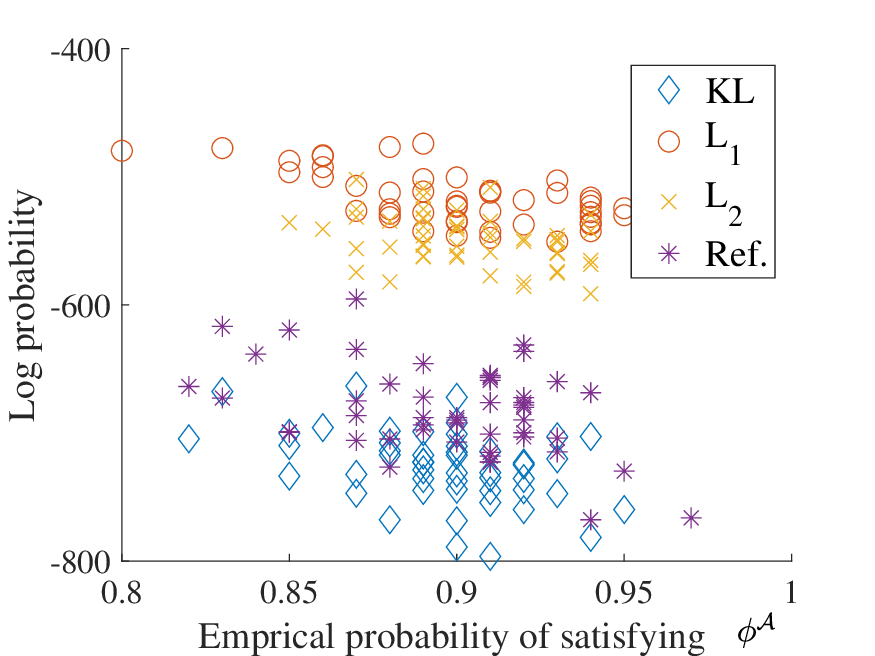}
		\label{fig:loglsf}
	}}

	\subfloat[Random MDP]{\scalebox{1}{
	\centering
\includegraphics[width=0.2\paperwidth]{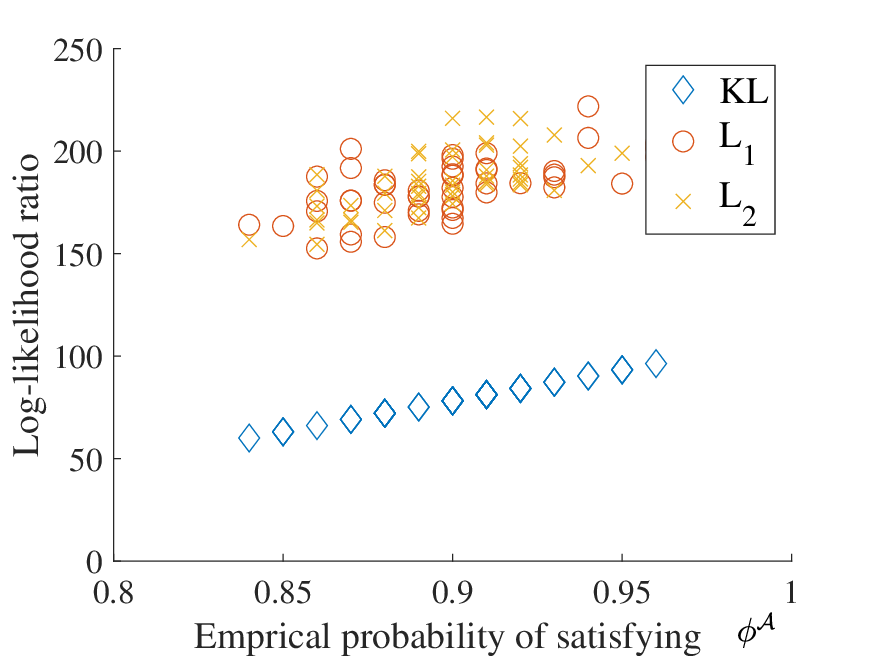}
		\label{fig:llrtrandom}
	}}
	\subfloat[MDP for San Francisco]{\scalebox{1}{
	\centering
\includegraphics[width=0.2\paperwidth]{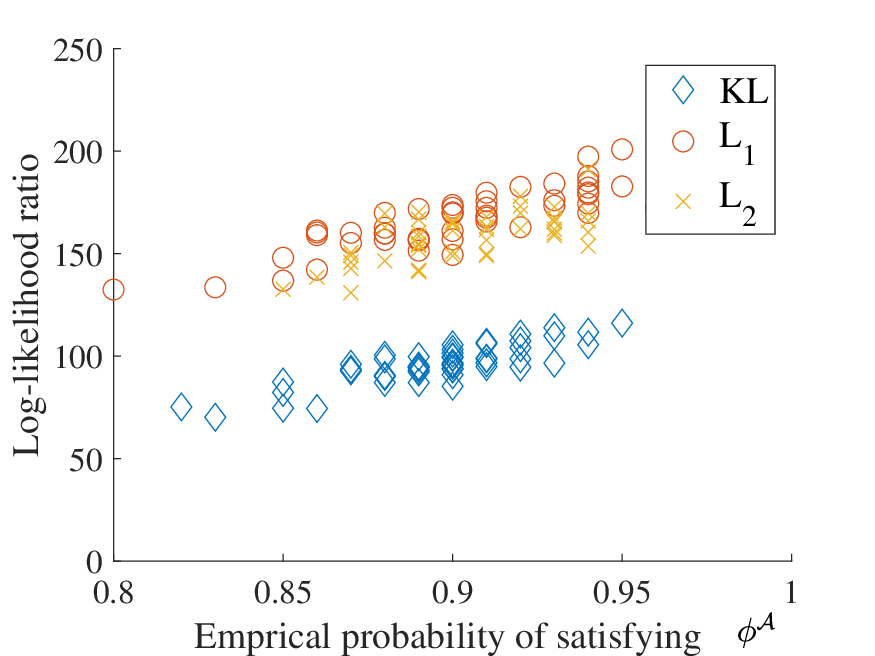}
		\label{fig:llrtsf}
	}}
	\caption{(a)-(b) The log-likelihoods under the reference policy. `Ref.' refers to the rare events of the reference policy that satisfies $\mathit{\specification^{\agent}}$ with high probability. `KL', `L$_{1}$', and `L$_{2}$' refer to the candidate deceptive policies. (c)-(d) The log-likelihood ratios compared with the reference policy. }
\label{fig:loglgraph}
\end{figure}

\subsection{Optimal Reference Policies} \label{subsection:admmexample}
We present an example of synthesis of optimal reference policies. The environment is a $4 \times 4$ grid world given in Figure \ref{fig:ADMMheatmaps} and is similar to the environment described in the example for the characteristics of deceptive policies. The green and red states are denoted with sets $g$ and $r$, respectively. At every state, there are $4$ available actions, namely, up, down, left, and right, at every state. When the agent takes an action the transition happens into the target direction with probability $0.7$ and in the other directions uniformly randomly with probability $0.3$. If a direction is out of the grid the transition probability to that direction is proportionally distributed to the other directions. The green state is absorbing and the initial state is the top-left state. 

The specification of the supervisor is to reach the green state, i.e., $ \lozenge g$. Note that the specification of the supervisor is satisfied with probability $1$ under any policy. The specification of the agent is to reach one of the red states, i.e., $\lozenge r$. The probability threshold for the agent's task is $0.3$. 

 We synthesize the reference policy via Algorithm \ref{algo:admm} given in Section \ref{subsection:ADMM}. In Algorithm \ref{algo:admm}, $z^{\svisor,k}$ represents the reference policy synthesized at iteration $k$. Similarly, $z^{\agent,k}$ represents the deceptive policy synthesized at iteration $k$. We plot the values of the KL divergences between these policies in Figure \ref{fig:ADMMgraph} and give the heatmaps for the occupation measures in Figure \ref{fig:ADMMheatmaps}. After few tens of iterations of the ADMM algorithm, the KL divergence value is near to the limit value which is $0.150$.  
 
 In Figure \ref{fig:ADMMgraph}, we also note that if the actual KL divergence value increases suddenly, the best response KL divergence value decreases. The reference policy tries to exploit suboptimal deceptive policies. While this exploitation increases the actual value, it causes suboptimality for the reference policy against the best deceptive policy.

 \begin{figure} [h]
	\centering
	\subfloat[]{\scalebox{1}{
			\centering
		\includegraphics[width=0.2\paperwidth]{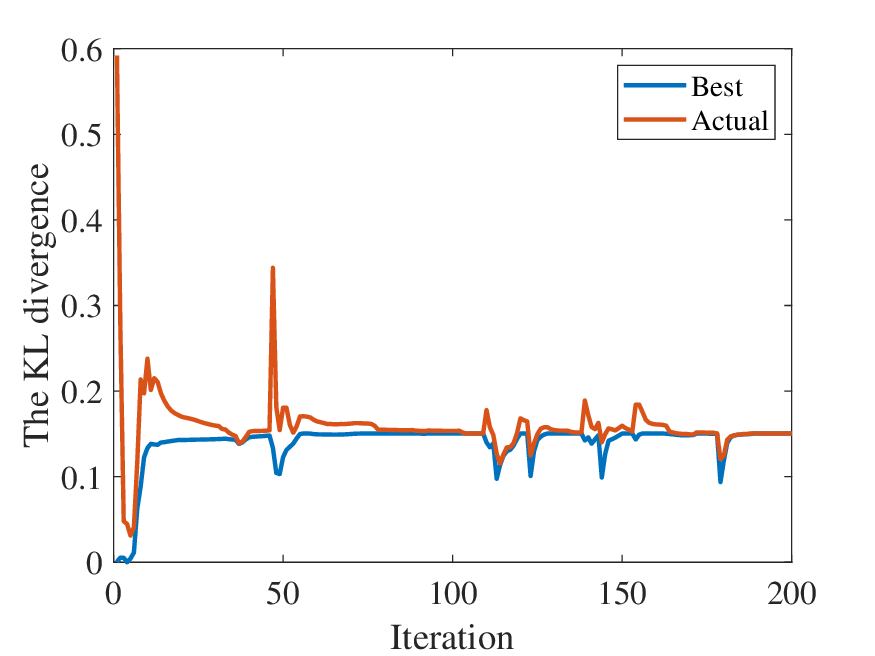}
	\label{fig:ADMMgraph}
	}}
	\subfloat[]{\scalebox{1}{
			\centering
		\includegraphics[width=0.2\paperwidth]{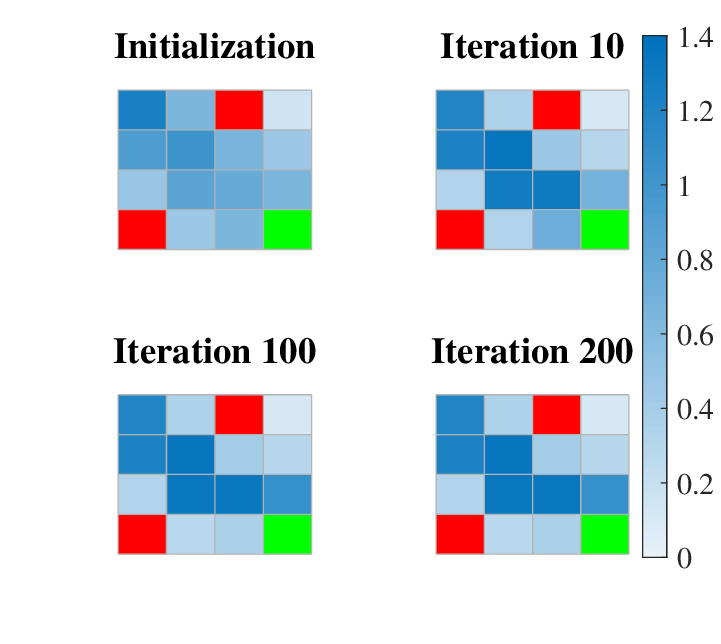}
		\label{fig:ADMMheatmaps}
	}}
\caption{(a) The KL divergence between the agent's policy and the reference policy. The curve ``Best'' refers to the case that the agent's policy is the best deceptive policy against the reference policy synthesized during the ADMM algorithm. The curve ``Actual'' refers to the case that the agent's policy is the policy synthesized during the ADMM algorithm. (b) Heatmaps of the occupation measures for the reference policy, i.e., $z^{\svisor, k}$ parameters of the Algorithm \ref{algo:admm}. The value of a state is the expected number of visits to the state.}
\end{figure}

The reference policy gradually gets away from the red states as shown in Figure \ref{fig:ADMMheatmaps}. Based on this observation, we expect that the relaxed problem given in Section \ref{subsection:relaxation} provides useful reference policies for the original problem. This expectation is indeed verified numerically: The reference policy synthesized via the relaxed problem, has a KL divergence of $0.150$, which is equal to the limit value of the ADMM algorithm.

\section{Conclusion} \label{section:conclusion}
We considered the problem of deception under a supervisor that provides a reference policy. We modeled the problem using MDPs and reachability specifications and proposed to use KL divergence for the synthesis of optimal deceptive policies. We showed that an optimal deceptive policy is stationary and its synthesis requires solving a convex optimization problem. We also considered the synthesis of optimal reference policies that easily prevent deception. We showed that this problem is NP-hard. We proposed a method based on the ADMM to compute a locally optimal solution and provided an approximation that can be modeled as a linear program.

In subsequent work we aim to extend the deception problem to a multi-agent settings where multiple malicious agents need to cooperate. Furthermore, it would be interesting to consider the case where a malicious agent first needs to detect the other malicious agents before cooperation. We also aim to study the scenario where the supervisor needs to learn the specification of the agent for the synthesis of the reference policy.

\begin{appendices}
    
 \section{Synthesis of Optimal Deceptive Policies under Nondeterministic Reference Policies} \label{section:nondeterministicref}
	In Problem \ref{problem:mindeviation}, we assume that the supervisor provides an explicit (possibly probabilistic) reference policy to the agent. It is possible that the reference policy is nondeterministic such that the supervisor disallows some actions at every state and the agent is allowed to take the other actions. We refer the interested readers to \cite{fard2011non} for the formal definition of the nondeterministic policies. A \textit{permissible policy} is a policy that takes an allowed action at every time step. The nondeterministic reference policy represents a set $\Pi^{\mathcal{S}}$ of policies that are permissible. The supervisor is indifferent between the policies in $\Pi^{\mathcal{S}}$. As in Section \ref{section:problemstatement}, we assume that the supervisor observes the transitions, but not the actions of the agent.
	
	We define the optimal deceptive policy as the policy that minimizes the KL divergence to any permissible policy subject to the task constraint of the agent. Under this definition, the optimal deceptive policy mimics the rare event that satisfies the agent's task, and that is most probable under one of the permissible policies. Formally, we solve the following problem for the synthesis of optimal deceptive policies under nondeterministic reference policies.
	
		\begin{problem}[Synthesis of Optimal Deceptive Policies under Nondeterministic Reference Policies]
		\label{problem:mindeviationnondet}
		Given an MDP $\mdp$, a reachability specification $\agentspec$, a probability threshold $\nu^{\agent}$, and a set $\Policy^{\svisor}$ of permissible policies, solve
		\begin{subequations}
			\label{problemeqn:mindeviationnondet}
			\begin{align}
			\underset{\substack{\policy^{\agent} \in \Policy(\mdp) \\ \policy^{\svisor} \in \Policy^{\svisor}}
			}{\inf} \quad
			&  KL\left(\Gamma^{\policy^{\agent}}_{\mdp} || \Gamma^{\policy^{\svisor}}_{\mdp} \right) \label{eq:pathprobmininfnondet}
			\\
			\text{\normalfont subject to } \quad
			& {\Pr }^{\policy^{\agent}}_{\mdp}(s_0 \models \agentspec) \geq \nu^{\agent}.
			\end{align}
		\end{subequations} If the optimal value is attainable, find a policy $\policy^{\agent}$ that is a solution to \eqref{problemeqn:mindeviationnondet}.
	\end{problem}

Similar to Assumption \ref{assumption:supervisorisstationary}, we assume that for every state in MDP $\mdp$, the set of allowed actions is fixed. For state $s \in \states$, we denote the set of allowed actions with $A^{\svisor}(s)$ and the possible successor state distributions with $\Lambda^{\svisor}(s)$. Under this assumption, we can identify the maximal end components of $\mdp$ for the permissible policies. If a maximal end component is closed, i.e., the maximum probability leaving the end component under the permissible policies is $0$, then the states of the maximal end component belongs to $C^{cl}$. If a maximal end component is open, i.e., the maximum probability leaving the end component under the permissible policies is $1$, then there exists an optimal policy that eventually leaves the end component since the agent can guarantee the same objective value by following a policy that leaves the end component and takes permissible actions after leaving the end component.

 While there exists an optimal policy that leaves the open end components eventually, the optimal policy may have infinite occupation measures at these states. We have the following assumption to ensure the boundedness of the occupation measures for these states. 
 \begin{assumption} \label{assumption:restimesarefinite}
 	For all $s \in \states\setminus (\closedset \cup \agentsset)$ and $a \in \actions(s)$, the occupation measure $x^{\agent}_{s,a}$ is upper bounded by $\theta$.
 \end{assumption}

As in the proof of Proposition \ref{proposition:stationaryisenough}, we can define a semi-infinite MDP with an optimal cost equal to the optimal value of Problem \ref{problem:mindeviationnondet}. In detail, at state $s \in S_{d}$, an action $a$ with successor state distribution $X_{s,a}$ has cost $\min_{X^{\svisor}_{s} \in \Lambda^{\svisor}_{s}} KL(X_{s,a} || X^{\svisor}_{s})$ where $X^{\svisor}_{s}$ is the distribution of successor states under the reference policy. We note that $\min_{X^{\svisor}_{s} \in \Lambda^{\svisor}_{s}} KL(X_{s,a} || X^{\svisor}_{s})$ is a convex function of $X_{s,a}$ since the Kullback-Leibler divergence is jointly convex in its arguments~\cite{boyd2004convex}. Since the occupation measures are bounded for all states in $\differset = \states \setminus (C_{cl} \cup C_{\agent})$ due to Assumption \ref{assumption:restimesarefinite} and the costs are convex functions of the policy parameters, there exists a stationary deterministic optimal policy for the semi-infinite MDP~\cite{altman1999constrained}. Consequently, there exists a stationary randomized optimal policy for Problem \ref{problem:mindeviationnondet} under Assumption \ref{assumption:restimesarefinite}.

Given that the optimal policy is stationary on $\mdp$, we compute the state-action occupation measures of the optimal policy by solving the following optimization problem:
\begin{subequations}
	\label{prog:stateprogramnondet}
	\begin{align}
	&\inf \quad \sum_{s \in \differset } \sum_{a \in \actions(s)} \sum_{q \in \successor(s)}  x^{\agent}_{s,a} \probs_{s,a,q} \nonumber
	\\& \hspace{0.5cm}\log \left( \frac{\sum_{a' \in \actions(s)} x^{\agent}_{s,a'} \probs_{s,a',q}}{ \left( \sum_{a \in A^{\svisor}(s)} \pi_{s,a} P_{s,a,q} \right) \left( \sum_{a' \in \actions(s)} x^{\agent}_{s,a'} \right)} \right) \label{eq:mininfrestimenondet}
	\\
	& \normalfont \text{subject to } \quad \nonumber
	\\ 
	&\eqref{cons:positiveactions} - \eqref{cons:reach1} \nonumber
	\\
	& \sum_{s \in \differset} \sum_{a \in \actions(s)} x^{\agent}_{s,a} \leq \theta,
	\\
	& \sum_{a \in A^{\svisor}(s)} \pi^{\svisor}_{s,a} = 1, \hspace{3.5cm} \forall s \in \differset 
	\end{align}
\end{subequations}
where the decision variables are $x^{\agent}_{s,a}$ for all $s \in \differset$ and $a \in \actions(s)$ and $\pi_{s,a}$ for all $s \in \differset$ and $a \in \actions^{S}(s)$. 

We remark that the objective function of \eqref{prog:stateprogramnondet} is a jointly convex function of $x^{\agent}_{s,a}$ and $\pi_{s,a}$ parameters. Having computed a set of optimal occupation measures, one can compute a stationary optimal deceptive policy.

  \color{black}{}
 
  \section{Proofs for the Technical Results} \label{section:proofs}
	We use the following definition and lemmas in the proof of Proposition \ref{proposition:expectedtimesarefinite}. We use ${\Pr}^{\policy}_{\mathcal{M}}(s \models \bigcirc\lozenge s)$ to denote the probability that $s$ is visited again from initial state $s$ under the stationary policy $\pi$ .
	
	\begin{definition}
		Let $Q$ be a probability distribution with a countable support $\mathcal{X}$. The \textit{entropy} of $Q$ is $H(Q) = -\sum_{x \in \mathcal{X}} Q(x) \log(Q(x)).$
	\end{definition}
	\begin{lemma}[Theorem 5.7 of \cite{conrad2004probability}] \label{lemma:geometricisthebest}
		Let $\mathcal{D}$ be the set of a distributions with support $\{ 1, 2, \ldots \}$ and the expected value of $c$. A random variable $X^{*} \sim Geo(1/c)$ maximizes $H(X)$ subject to $X \in \mathcal{D}$ where $H(X^{*}) = c \left( - \frac{1}{c} \log\left(  \frac{1}{c}\right) - \left(1- \frac{1}{c} \right) \log\left(1- \frac{1}{c} \right) \right) = c H\left(Ber\left(\frac{1}{c}\right)\right)$. 
	\end{lemma}
	
	\begin{lemma} \label{lemma:statevisittimes}
		Consider an MDP $\mathcal{M} = (\states, \actions, \probs, \atomicprops, L)$. Let $N^{\policy}_s$ denote the number of visits to the state $s$ under a stationary policy $\policy$ such that $\mathbb{E}[N^{\policy}_s] < \infty$. $N^{\policy}_{s}$ satisfies $\Pr(N^{\policy}_s = 0) = {\Pr}^{\policy}_{\mathcal{M}}(s_0 \not\models \lozenge s)$ and ${\Pr}(N^{\policy}_s = i) = {\Pr}^{\policy}_{\mathcal{M}}(s_{0} \models \lozenge s) {\Pr}^{\policy}_{\mathcal{M}}(s \models \bigcirc\lozenge s)^{i-1} {\Pr}^{\policy}_{\mathcal{M}}(s \not\models \bigcirc\lozenge s).$
	\end{lemma}

	\begin{proof}[\textbf{Proof of Proposition \ref{proposition:expectedtimesarefinite}}]
		We prove this proposition by contradiction. We first provide a lower bound for the objective function of Problem \ref{problem:mindeviation}. Then, we show that as the state-action occupation measures approach to infinity, the lower bound approaches to infinity. Hence, the state-action occupation measures must be bounded in order to have a finite value for the objective function of Problem \ref{problem:mindeviation}.
		
		Let $d^{*}$ be the optimal value of Problem \ref{problem:mindeviation}. For a state $s \in S \setminus \closedset$, first consider the case $\Pr^{\policy^{\svisor}}_{\mathcal{M}}(s_{0} \models \lozenge s) = 0$, i.e., $s$ is unreachable under $\policy^{\svisor}$. In this case, the agent's policy $\policy^{\agent}$ must satisfy $\Pr^{\policy^{\agent}}_{\mathcal{M}}(s_{0} \models \lozenge s) = 0$, i.e., $s$ must be unreachable under $\policy^{\agent}$, otherwise the KL divergence is infinite. Hence the occupation measure is zero in this case.
		
		Consider $\Pr^{\policy^{\svisor}}_{\mathcal{M}}(s_{0} \models \lozenge s) > 0$. For this case, we will show that if the occupation measure is greater than some finite value, then the KL divergence between the path distributions is greater than $d^{*}$. Denote the number visits to $s$ with $N^{\policy^{\agent}}_{s}$ and $N^{\policy^{\svisor}}_{s}$ under $\policy^{\agent}$ and $\policy^{\svisor}$, respectively. We have the following claim: Given $\Pr^{\policy^{\svisor}}_{\mathcal{M}}(s_{0} \models \lozenge s) > 0$, $\Pr^{\policy^{\svisor}}_{\mathcal{M}}(s \models \bigcirc \lozenge s) \in [0,1)$, and $d^{*} >0$, there exists an $M_{s}$ such that for all $\policy^{\agent}$ that satisfies $\mathbb{E}[N^{\policy^{\agent}}_{s}] > M_{s}$, we have $KL(\Gamma^{\policy^{\agent}}_{\mathcal{M}_p} || \Gamma^{\policy^{\svisor}}_{\mathcal{M}_p}) > d^{*}$. 
		
		We consider a partitioning of paths according to the number of times $s$ appears in a path. By the data processing inequality given in \eqref{ineq:dataprocessing}, we have that $KL(\Gamma^{\policy^{\agent}}_{\mathcal{M}_p} || \Gamma^{\policy^{\svisor}}_{\mathcal{M}_p}) \geq KL(N^{\policy^{\agent}}_{s} || N^{\policy^{\svisor}}_{s})$, i.e., the KL divergence between the path distributions is lower bounded by the KL divergence between the distributions of number visits to $s$. Therefore it suffices to prove the following claim: Given $\Pr^{\svisor}_{\mathcal{M}}(s_0 \models \lozenge s) > 0$, $\Pr^{\svisor}_{\mathcal{M}}(s \models \bigcirc \lozenge s) \in [0,1)$, and $d^{*} >0$, there exists an $M_{s}$ such that for all $\policy^{\agent}$ that satisfies $\mathbb{E}[N^{\policy^{\agent}}_{s}] > M_{s}$, we have $KL(N^{\policy^{\agent}}_{s} || N^{\policy^{\svisor}}_{s}) > d^{*}$. 
		
		Define a random variable $\hat{N}^{\policy^{\agent}}_{s}$ such that $\Pr(\hat{N}^{\policy^{\agent}}_{s} = i) = \Pr(N^{\policy^{\agent}}_{s} = i | N^{\policy^{\agent}}_{s} > 0).$ For notational convenience denote $r^{\svisor} = 1 - \Pr^{\svisor}_{\mathcal{M}}(s_0 \models \lozenge s)$, $l^{\svisor} = \Pr^{\svisor}_{\mathcal{M}}(s \models \bigcirc \lozenge s)$, $p_{i} = \Pr(N^{\policy^{\agent}}_{s} = i)$ and $\hat{p}_{i} = \Pr(\hat{N}^{\policy^{\agent}}_{s} = i)$. Also define $ M^{\agent}_{s} := \mathbb{E}[N^{\policy^{\agent}}_{s}]$, $\hat{M}^{\agent}_{s}:= \mathbb{E}[\hat{N}^{\policy^{\agent}}_{s}] = \frac{M^{\agent}_{s}}{1 - p_{0}}$, and $M^{\svisor}_{s} :=\mathbb{E}[N^{\policy^{\svisor}}_{s}] $. 
		
		We want to show that $M^{\agent}_{s}$ is bounded for a finite $d^{*}$. Assume that $M^{\agent}_{s} \leq M^{\svisor}_{s}$. In this case the $M^{\agent}_{s}$ is finite since $M^{\svisor}_{s}$ is finite. If $M^{\agent}_{s} > M^{\svisor}_{s}$, we have
		\small
		\begin{subequations}
			\begin{align}
			&KL(N^{\policy^{\agent}}_{s} || N^{\policy^{\svisor}}_{s}) 
			\\
			&= p_{0} \log \left( \frac{p_{0}}{r^{\svisor}}\right) + \sum_{i=1}^{\infty} p_{i} \log \left( \frac{p_{i}}{(1-r^{\svisor}) (l^{\svisor})^{i-1} (1 - l^{\svisor})}\right) \label{eq:KLstateopenedup}
						\\
			&= p_{0} \log \left( \frac{p_{0}}{r^{\svisor}}\right) +  \sum_{i=1}^{\infty}(1-p_{0}) \hat{p}_{i} \log \left( \frac{1-p_{0}}{1 - r^{\svisor}}\right) \nonumber
			\\
			&\quad  + \sum_{i=1}^{\infty} (1 - p_{0})\hat{p}_{i} \log \left( \frac{\hat{p}_{i}}{ (l^{\svisor})^{i-1} (1 - l^{\svisor})}\right)
			\\
			&= p_{0} \log \left( \frac{p_{0}}{r^{\svisor}}\right) +  (1-p_{0}) \log \left( \frac{1-p_{0}}{1 - r^{\svisor}}\right) \nonumber
			\\
			&\quad  + \sum_{i=1}^{\infty} (1 - p_{0})\hat{p}_{i} \log \left( \frac{\hat{p}_{i}}{ (l^{\svisor})^{i-1} (1 - l^{\svisor})}\right)
			\\
			&\geq (1- p_{0})\sum_{i=1}^{\infty}\hat{p}_{i} \log \left( \frac{\hat{p}_{i}}{ (l^{\svisor})^{i-1} (1 - l^{\svisor})}\right) \label{eq:removedtermsarepositive}
			\\
			&= (1 - p_{0})\sum_{i=1}^{\infty}\hat{p}_{i} \log \left( \hat{p}_{i} \right) \nonumber
			\\
			&- (1 - p_{0})\sum_{i=1}^{\infty}\hat{p}_{i} \log \left( (l^{\svisor})^{i-1} (1 - l^{\svisor}) \right) 
			\\
			&=-( 1 -p_{0}) H(\hat{N}^{\policy^{\agent}}_{s}) - (1 - p_{0})\sum_{i=1}^{\infty}\hat{p}_{i} \log \left( (l^{\svisor})^{i-1} (1 - l^{\svisor}) \right)
			\end{align}
		\end{subequations} \normalsize
	where the equality \eqref{eq:KLstateopenedup} follows from Lemma \ref{lemma:statevisittimes}. The inequality in \eqref{eq:removedtermsarepositive} holds since the removed terms correspond to $KL(Ber(p_{0}) || Ber(r^{\svisor}))$ which is nonnegative.
	
	By using Lemma \ref{lemma:geometricisthebest} to upper bound $H(\hat{N}^{\policy^{\agent}}_{s})$ and the definitions  we have the following inequality.
	\small
	\begin{subequations}
		\begin{align}
			&KL(N^{\policy^{\agent}}_{s} || N^{\policy^{\svisor}}_{s}) \nonumber
			\\
			&\geq (p_{0} -1)\left( H(\hat{N}^{\policy^{\agent}}_{s}) + \sum_{i=1}^{\infty}\hat{p}_{i} \log \left( (l^{\svisor})^{i-1} (1 - l^{\svisor}) \right) \right)
			\\
			&\geq (p_{0} -1) \left(\hat{M}^{\agent}_{s} H \left( Ber\left( \frac{1}{\hat{M}^{\agent}_{s}}\right) \right) \right.  \nonumber
			\\
			&\quad \left. + \sum_{i=1}^{\infty}\hat{p}_{i} \log \left( (l^{\svisor})^{i-1} (1 - l^{\svisor}) \right)  \right) \label{eq:entropyupperbound}
				\\
			&= (p_{0} -1) \left(\hat{M}^{\agent}_{s} H \left( Ber\left( \frac{1}{\hat{M}^{\agent}_{s}}\right) \right) \right.  \nonumber
			\\
			&\quad \left. + \sum_{i=1}^{\infty}\hat{p}_{i} \left( (i-1) \log  (l^{\svisor})  +  \log(1 - l^{\svisor}) \right)  \right)
			\\
			&= (p_{0} -1) \left(\hat{M}^{\agent}_{s} H \left( Ber\left( \frac{1}{\hat{M}^{\agent}_{s}}\right) \right) \right.  \nonumber
			\\
			&\quad \left. + \left( \log(1 - l^{\svisor})  + \left( \hat{M}^{\agent}_{s} - 1\right) \log ( l^{\svisor} )   \right)  \right)
			\\
			&= (1- p_{0}) \left(- \hat{M}^{\agent}_{s} H \left( Ber\left( \frac{1}{\hat{M}^{\agent}_{s}}\right) \right) \right.  \nonumber
			\\
			&= M^{\agent}_{s} \left( KL\left( Ber\left(\frac{1}{\hat{M}^{\agent}_{s}}\right) || Ber(1 - l^{\svisor}) \right) \right)
			\end{align}
		\end{subequations}
	\normalsize

Now assume that $M^{\agent}_{s} \geq \frac{c}{1 - l^{\svisor}}$ where $c>1$ is a constant. In this case, we have
			\begin{subequations} \small
	\begin{align}
	&KL(N^{\policy^{\agent}}_{s} || N^{\policy^{\svisor}}_{s})
	\\
	&\geq M^{\agent}_{s} \left( KL\left( Ber\left(\frac{1}{\hat{M}^{\agent}_{s}}\right) || Ber(1 - l^{\svisor}) \right) \right)
	\\
	&\geq M^{\agent}_{s} \left( KL \left( Ber \left( \frac{1}{M^{\agent}_{s}} \right) || Ber \left(1 - l^{\svisor} \right) \right) \right) \label{eq:lowerbound}
	\\
	&\geq M^{\agent}_{s} \left( KL \left( Ber \left( \frac{1 - l^{\svisor}}{c} \right) || Ber \left(1 - l^{\svisor} \right) \right) \right) 
	\end{align}
\end{subequations} \normalsize
since $\hat{M}^{\agent}_{s} > M^{\agent}_{s}$ and for a variable $x$ such that $x \geq \frac{1}{1 - l^{\svisor}}$, the value of $KL(Ber(\frac{1}{x}) || Ber(1 - l^{\svisor}))$ is increasing in $x$. 

Note that $KL \left( Ber \left( \frac{1 - l^{\svisor}}{c} \right) || Ber \left(1 - l^{\svisor} \right) \right)$ is a positive constant. We can easily see that there exists an $M_{s}$ such that $KL(N^{\policy^{\agent}}_{s} || N^{\policy^{\svisor}}_{s}) > d^{*}$ if $M^{\agent}_{s} > M_s$.

We proved that for a given constant, for every transient state of the supervisor the occupancy measure under the agent's policy must be bounded by some constant otherwise the KL divergence between distributions for the number of states to this state is greater than the constant. Since the KL divergence between the path distributions is lower bounded by the KL divergence for states, the finiteness of the KL divergence between the path distributions implies that the occupancy measure under the agent's policy for every transient state of the supervisor. 

Thus, if the optimal value of Problem \ref{problem:mindeviation} is finite, the occupation measures under $\policy^{\agent}$ must be bounded by some $M_{s} < \infty$ for all $s \in S \setminus \closedset$.
	\end{proof}

	We use the following definition in the proof of Lemma \ref{lemma:globalislocal}. We remark that the proof of Lemma \ref{lemma:globalislocal} is fairly similar with the proof of Lemma 2 from \cite{savas2018entropy}. 
	\begin{definition}
		A \textit{k-length path fragment} $\xi = s_0 s_1 \ldots s_{k}$ for an MDP $\mdp$ is a sequence of states under policy $\policy= \mu_0\mu_1\ldots$ such that $\sum_{a \in \actions(s_t)} \probs(s_t,a,s_{t+1}) \mu_t(s_t,a) > 0$ for all $k > t\geq0$. The distribution of k-length path fragments for $\mdp$ under policy $\policy$ is denoted by $\Gamma^{\policy}_{\mdp,k}$.
	\end{definition}

	\begin{lemma} \label{lemma:globalislocal} 
		The KL divergence $KL(\Gamma^{\policy^{\agent}}_{\mdp,k} || \Gamma^{\policy^{\svisor}}_{\mdp,k})$ between the distributions of $k$-length path fragments for stationary policies $\policy^{\agent}$ and $\policy^{\svisor}$ is equal to the expected sum of KL divergences between the successor state distributions of $\policy^{\agent}$ and $\policy^{\svisor}$ that is
		\begin{align} \label{eq:globalislocal}
			& \sum_{t = 0}^{k-1} \sum_{s \in \differset} {\Pr}^{\policy^{\agent}}(s_t = s) \sum_{q \in \successor(s)} \sum_{a \in \actions(s)} \probs_{s,a,q} \policy^{\agent}_{s,a} \nonumber
			\\
			& \log\left( \frac{\sum_{a' \in \actions(s)} \probs_{s,a',q} \policy^{\agent}_{s,a'}}{ \sum_{a' \in \actions(s)} \probs_{s,a',q} \policy^{\svisor}_{s,a}}\right). \nonumber
		\end{align}
		Furthermore, if $KL(\Gamma^{\policy^{\agent}}_{\mdp} || \Gamma^{\policy^{\svisor}}_{\mdp})$ is finite, it is equal to
$$\sum_{s \in \differset} \sum_{q \in \differset} \sum_{a \in \actions(s)} \probs_{s,a,q} x^{\agent}_{s,a} \log\left( \frac{\sum_{a' \in \actions(s)} \probs_{s,a',q} x^{\agent}_{s,a'}}{ \policy^{\svisor}_{s,q} \sum_{a' \in \actions(s)} x^{\agent}_{s,a'}}\right)$$

	\end{lemma}
	\begin{proof}[\textbf{Proof of Lemma \ref{lemma:globalislocal}}]
		For MDP $\mathcal{M}$, denote the set of $k$-length path fragments by $\Xi_k$ and the probability of the $k$-length path fragment $\xi_k = s_0 s_1 \ldots s_{k}$ under the stationary policy $\policy$ by $\Pr^{\policy}(\xi_k)$. We have ${\Pr}^{\policy}(\xi_k) = \prod_{t=0}^{k-1} \sum_{a \in \actions(s_t)} \probs_{s_t,a,s_{t+1}} \policy_{s_t, a}.$ Consequently, we have 
		\small
		\begin{subequations}
			\begin{align*}
			&KL(\Gamma^{\policy^{\agent}}_{\mdp,k} || \Gamma^{\policy^{\svisor}}_{\mdp,k})  =\sum_{\xi_k \in \Xi_k} {\Pr}^{\policy^{\agent}}(\xi_k) \log \left( \frac{{\Pr}^{\policy^{\agent}}(\xi_k)}{{\Pr}^{\policy^{\svisor}}(\xi_k)} \right) 
			\\
			& = \sum_{t = 0}^{k-1}  \sum_{\xi_k \in \Xi_k} {\Pr}^{\policy^{\agent}}(\xi_k) \log \left( \frac{\sum_{a' \in \actions(s_t)} \probs_{s_t,a',s_{t+1}} \policy^{\agent}_{s_t, a'}}{\sum_{a' \in \actions(s_t)} \probs_{s_t,a',s_{t+1}} \policy^{\svisor}_{s_t, a'}} \right)
			\\
			& = \sum_{t = 0}^{k-1} \sum_{\xi_k \in \Xi_k} {\Pr}^{\policy^{\agent}}(\xi_k) \sum_{s \in \differset} \mathds{1}_{s}(s_t)  \\
			& \quad \sum_{q \in \successor(s)} \mathds{1}_{q}(s_{t+1} | s_t = s) \log \left( \frac{\sum_{a' \in \actions(s_t)} \probs_{s,a',q} \policy^{\agent}_{s,a'}}{\sum_{a' \in \actions(s_t)} \probs_{s,a',q} \policy^{\svisor}_{s,a'}} \right)		
			\\
			& = \sum_{t = 0}^{k-1} \sum_{s \in \differset}{\Pr}^{\policy^{\agent}}(s_t = s) \sum_{q \in \successor(s)} \sum_{a \in \actions(s_t)} 
			\\
			& \quad \probs_{s,a,q} \policy^{\agent}_{s,a'}  \log \left( \frac{\sum_{a' \in \actions(s_t)} \probs_{s,a',q} \policy^{\agent}_{s,a'}}{\sum_{a' \in \actions(s_t)} \probs_{s,a',q} \policy^{\svisor}_{s,a'}} \right)	
			\end{align*}
		\end{subequations}
	
		\normalsize
		
		If $KL(\Gamma^{\policy^{\agent}}_{\mdp} || \Gamma^{\policy^{\svisor}}_{\mdp})$ is finite, we have \small
	\begin{subequations}
		\begin{align*}
		&KL(\Gamma^{\policy^{\agent}}_{\mdp} || \Gamma^{\policy^{\svisor}}_{\mdp}) = \lim\limits_{k \to \infty} KL(\Gamma^{\policy^{\agent}}_{k} || \Gamma^{\policy^{\svisor}}_{k})
		\\
		&=\lim\limits_{k \to \infty} \sum_{s \in \differset}\sum_{q \in \successor(s)} \sum_{a \in \actions(s_t)} \sum_{t = 0}^{k-1} 
		\\
		& \quad {\Pr}^{\policy^{\agent}}(s_t = s) \probs_{s,a,q} \policy^{\agent}_{s,q}\log\left( \frac{\sum_{a' \in \actions(s_t)} \probs_{s,a',q} \policy^{\agent}_{s,q}}{\sum_{a' \in \actions(s_t)} \probs_{s,a',q}\policy^{\svisor}_{s,q}}\right)
		\\
		&=\sum_{s \in \differset} \sum_{q \in \successor(s)} \sum_{a \in \actions(s)} 
		\\
		& \probs_{s,a,q} x^{\agent}_{s,a} \log\left( \frac{\sum_{a' \in \actions(s)} \probs_{s,a',q} x^{\agent}_{s,a'}}{ \policy^{\svisor}_{s,q} \sum_{a' \in \actions(s)} x^{\agent}_{s,a'}}\right).
		\end{align*}
	\end{subequations} \normalsize
	
	Finally, since $P_{s,a,q}$ is zero for all $q \not\in Succ(s)$ and we defined $0 \log 0 = 0$, we can safely replace $Succ(s)$ with $S$. 
	\end{proof}

	\begin{proof}[\textbf{Proof of Proposition \ref{proposition:optimalisachieved}}]
		Assume that $KL(\Gamma^{\policy^{\agent}}_{\mathcal{M}} || \Gamma^{\policy^{\svisor}}_{\mathcal{M}})$ is finite under the stationary policies $\policy^{\agent}$ and $\policy^{\svisor}$. The objective function of the problem given in \eqref{problemeqn:mindeviation} is equal to $$\sum_{s \in \differset} \sum_{q \in \successor(s)} \sum_{a \in \actions(s)} \probs_{s,a,q} x^{\agent}_{s,a} \log\left( \frac{\sum_{a' \in \actions(s)} \probs_{s,a',q} x^{\agent}_{s,a'}}{ \policy^{\svisor}_{s,q} \sum_{a' \in \actions(s)} x^{\agent}_{s,a'}}\right)$$ due to Lemma \ref{lemma:globalislocal}.		The constraints \eqref{cons:positiveactions}-\eqref{cons:floweqn} define the stationary policies that make the states in $\differset$ have valid and finite occupation measures and the constraint \eqref{cons:reach1} encodes the reachability constraint. 
		
		Note that $$\sum_{q \in \states} \sum_{a \in \actions(s)} \probs_{s,a,q} x^{\agent}_{s,a} \log\left( \frac{\sum_{a' \in \actions(s)} \probs_{s,a',q} x^{\agent}_{s,a'}}{ \policy^{\svisor}_{s,q} \sum_{a' \in \actions(s)} x^{\agent}_{s,a'}}\right)$$ is the KL divergence between $\left[\sum_{a \in \actions(s)} \probs_{s,a,q} x^{\agent}_{s,a}\right]_{q \in \successor(s)}$ and $\left[\policy^{\svisor}_{s,q} \sum_{a' \in \actions(s)} x^{\agent}_{s,a'}\right]_{q \in \successor(s)}$, which is convex in $x^{\agent}_{s,a}$ variables. Since the objective function of \eqref{prog:stateprogram} is a sum of convex functions and the constraints are affine, \eqref{prog:stateprogram} is a convex optimization problem.   
		
		We now show that there exists a stationary policy on $\mdp$ that achieves the optimal value of \eqref{problem:mindeviation}. By Proposition \ref{proposition:expectedtimesarefinite}, we have that for all $s \in \differset$, the occupation measures must be bounded. We may apply the constraints $x^{\agent}_{s,a} \leq M_{s}$ for all $s$ in $\differset$ and $a$ in $\actions(s)$ without changing the optimal value of \eqref{prog:stateprogram}. After this modification, since the objective function is a continuous function of $x^{\agent}_{s,a}$ values and the feasible space is compact, there exists a set of occupation measure values, and consequently a stationary policy that achieves the optimal value of \eqref{prog:stateprogram}.
	\end{proof}

\begin{proof}[\textbf{Proof of Proposition \ref{proposition:supervisorattains}}]
	The condition $\probs_{s,a,q} > 0$ for all $s \in \states_{d}$, $a \in \actions(s)$, and $q \in Succ(s)$ implies that $\sum_{a \in \actions(s)} x^{\svisor}_{s,a} \probs_{s,a,q}$ is strictly positive for all $q \in \successor(s)$. Note that for the states $q \not\in Succ(s)$, we have $\sum_{a \in \actions(s)} x^{\agent}_{s,a} P(s,a,q)=0$. We also note that by Assumption \ref{assumption:copisfixed}, the occupation measures are bounded for all $s \in \states_{d}$ under $\policy^{\svisor}$. Hence, the objective function of \eqref{prog:stateprogramsupervisor} is bounded and jointly continuous in $x^{\svisor}_{s,a}$ and $x^{\agent}_{s,a}$.

	Since in we showed that there exists a policy that attains the optimal value of Problem \ref{problem:mindeviation}, we may represent the optimization problem given in \eqref{prog:stateprogramsupervisor} as $$\underset{{x^{\svisor}}}{\sup} \ \underset{x^{\agent}}{\min} \ f(x^{\svisor}, x^{\agent})$$ subject to $x^{\svisor} \in X^{\svisor}$ and $x^{\agent} \in X^{\agent}$. Note that $X^{\svisor}$ and $X^{\agent}$ are compact spaces, since the occupation measures are bounded for all state-action pairs. Given that $X^{\agent}$ is a compact space, the function $f'(x^{\svisor}) = {\min}_{x^{\agent}} f(x^{\svisor}, x^{\agent})$ is a continuous function of $x^{\svisor}$ \cite{clarke1975generalized}. The optimal value of ${\sup}_{{x^{\svisor}}}  \ f'(x^{\svisor})$ is attained. Consequently, there exists a policy $\pi^{\svisor}$ that achieves the optimal value of \eqref{prog:stateprogramsupervisor}.
\end{proof}

\color{blue}

\color{black}

\end{appendices}

\bibliography{ref2}	

\begin{IEEEbiography}[{\includegraphics[width=1in,height=1.25in,clip,keepaspectratio]{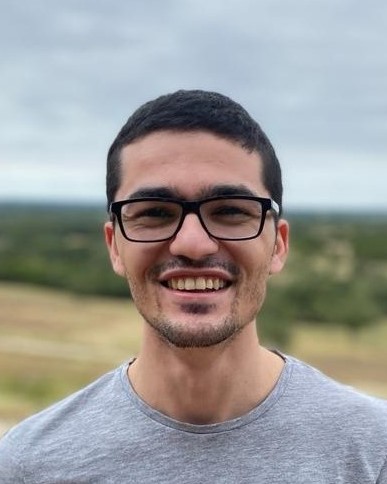}}]{Mustafa O. Karabag} joined the Department of Electrical and Computer Engineering at the University of Texas at Austin as a Ph.D. student in Fall 2017. He received his B.S. degree in Electrical and Electronics Engineering from Bogazici University in 2017. His research focuses on developing theory and algorithms for non-inferable, deceptive planning in adversarial environments.
\end{IEEEbiography}

\begin{IEEEbiography}[{\includegraphics[width=1in,height=1.25in,clip,keepaspectratio]{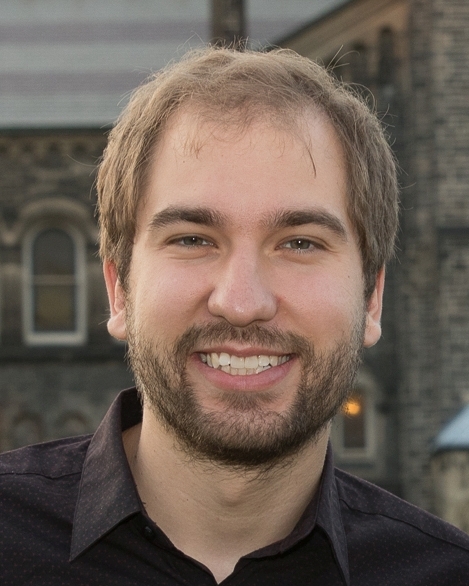}}] {Melkior Ornik} is an assistant professor in the Department of Aerospace Engineering and the Coordinated Science Laboratory at the University of Illinois at Urbana-Champaign. He received his Ph.D. degree from the University of Toronto in 2017. His research focuses on developing theory and algorithms for learning and planning of autonomous systems operating in uncertain, complex and changing environments, as well as in scenarios where only limited knowledge of the system is available.
\end{IEEEbiography}

\begin{IEEEbiography}[{\includegraphics[width=1in,height=1.25in,clip,keepaspectratio]{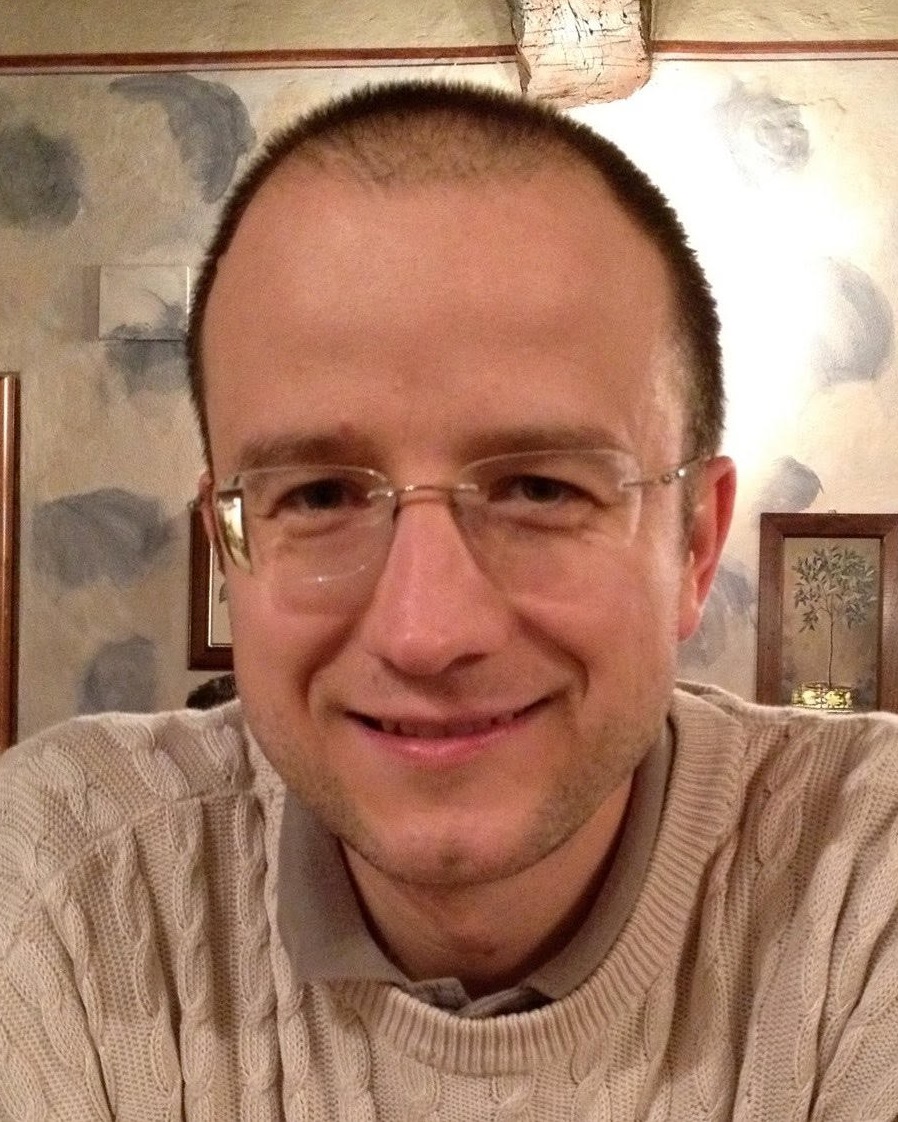}}]{Ufuk Topcu} is an associate professor at The University of Texas at Austin. He received his Ph.D. from the University of California, Berkeley in 2008. Ufuk held a postdoctoral research position at California Institute of Technology until 2012 and was a research assistant professor at the University of Pennsylvania until 2015. He is the recipient of the Antonio Ruberti Young Researcher Prize, the NSF CAREER Award and the Air Force Young Investigator Award. His research focuses on the design and verification of autonomous systems through novel connections between formal methods, learning theory and controls.
\end{IEEEbiography}

\newpage
\onecolumn
\section*{Supplementary Material for ``Deception in Supervisory Control''}

	This supplementary material contains the proof for NP-hardness of the synthesis of optimal reference policies (Proposition \ref{proposition:hardness}).

	\subsection{The set partition problem and the NP-hardness of the linear multiplicative programming}
The set partition problem~\cite{garey1979computers,karp1972reducibility} is the following:

\textbf{Instance: } An $m \times n$ $0-1$matrix $M$ satisfying $n > m$.

\textbf{Question: } Is there a $0-1$ vector $x$ satisfying 	$	 \sum_{\substack{j=1 \\ M_{ij} = 1}}^{n} x_{j} = 1$ for all $i \in [m]$.

The set partition problem is NP-hard.

Matsui \cite{matsui1996np} showed that the linear multiplicative programming is NP-hard. Let $M$ be an $m \times n$ $0-1$ matrix with $n \geq m$ and $n \geq 5$, and $p = n^{n^{4}}$. The problem 
\small
\begin{subequations}
	\label{matsuiprogramsupp}
	\begin{align}
	&\min \quad && (2p^{4n} - p + 2p^{2n} x_{0} + y_{0}) (2p^{4n} - p - 2p^{2n} x_{0} + y_{0}) 
	\\
	&\normalfont \text{subject to} && x_{i} \geq 0, \ \forall i \in [n]
	\\
	& && y_{ij} \geq 0, \ \forall i ,j \in [n]
	\\
	& \forall i\in [n], && x_{i} \leq 1
	\\
	&\forall i,j\in [n], && y_{ij} \leq 1
	\\
	&  \substack{\forall i,j \in [n] \\ i\neq j}, && x_{i}  \geq y_{ij}, 
	\\
	&  \substack{\forall i,j \in [n] \\ i\neq j}, && x_{j}  \geq y_{ij} , 
	\\
	&  \substack{\forall i,j \in [n] \\ i\neq j}, &&  y_{ij}  \geq x_{i} + x_{j} - 1, 
	\\
	& \forall i \in [n], &&  y_{ii} = x_{i},
	\\
	&\forall i \in [m], && \sum_{\substack{j=1 \\ M_{ij} = 1}}^{n} x_{j} = 1, \label{consMsupp}
	\\
	& && x_{0} = \sum_{i=1}^{n} p^{i} x_{i}
	\\
	& && y_{0} = \sum_{i=1}^{n} \sum_{j=1}^{n} p^{i+j} y_{ij} 
	\end{align}
\end{subequations}
\normalsize
where the decision variables are $x_{i}$ for all $i \in [n]$ and $y_{ij}$ for all $i,j \in [n]$ is NP-hard. 

In \cite{matsui1996np}, it is proved that the optimal value of \eqref{matsuiprogramsupp} is less than or equal to $4p^{8n}$ if and only if there exists a $0-1$ solution for $x_{1}, \ldots, x_{n}$ satisfying \eqref{consMsupp}. Since the decision problem of \eqref{matsuiprogramsupp} correspond to solving the set partition problem, \eqref{matsuiprogramsupp} is NP-hard.  

\subsection{The proof for the NP-hardness of the synthesis of optimal reference policies}

\begin{proof}[\textbf{Proof of Proposition 5}]
	We show the NP-hardness of Problem 2 by creating an instance of it that is equivalent to \eqref{matsuiprogramsupp}. The proof relies on the explicit construction of an MDP and specifications. 
	
	Let $M$ be an $m \times n$ $0-1$ matrix with $n \geq m$ and $n \geq 5$ and $d_{i} = \sum_{j=1}^{n} M_{ij}$ for all $i \in m$. Also let $p = n^{n^{4}}$, $c_{1} = \sum_{i=1}^{n} p^{i}$, $c_{2} = \sum_{i=1}^{n} \sum_{j=1}^{n} p^{i+j},$ and $C = 2p^{2n}c_{1} + c_{2} + 2p^{4n} - p.$

	Let $\mdp = (\states, \actions, \probs, \initialstate)$ be an MDP where $\states = \lbrace s_{0}, s^{+}_{\gamma}, q ,w_{1}, w_{2}\rbrace \cup \lbrace s^{+}_{i}, s^{-}_{i} | i \in [n] \rbrace \cup \lbrace v^{+}_{ij}, v^{-}_{ij} | i,j \in [n] \rbrace$, $\actions(s_{0}) = \lbrace \gamma, \delta \rbrace \cup \lbrace \alpha_{i} | i \in [n] \rbrace \cup \lbrace \beta_{ij} | i,j \in [n] \rbrace$ and $\actions(s) = \lbrace \epsilon \rbrace$ for all $s \in \states \setminus \lbrace s_{0} \rbrace$. The trasition probability function $\probs$ is defined as follows
	\begin{itemize}
		\item $\probs_{\initialstate, \alpha_{i}, s^{-}_{j}}=\frac{1}{2n^{2} + 2n + 2}$ for all $i,j \in [n]$ and $i\neq j$,
		\item $\probs_{\initialstate, \alpha_{i}, s^{-}_{i}}=0$ for all $i \in [n]$, 
		\item $\probs_{\initialstate, \alpha_{i}, s^{+}_{i}}=\frac{1}{2n^{2} + 2n + 2}$ for all $i \in [n]$, 
		\item $\probs_{\initialstate, \alpha_{i}, s^{+}_{j}}=0$ for all $i,j \in [n]$ and $i \neq j$,
		\item $\probs_{\initialstate, \alpha_{i}, v^{-}_{jk}}=\frac{1}{2n^{2} + 2n + 2}$ for all $i,j,k \in [n]$, 
		\item $\probs_{\initialstate, \alpha_{i}, v^{+}_{jk}}=0$ for all $i,j,k \in [n]$, 
		\item $\probs_{\initialstate, \alpha_{i}, w_{1}}=\frac{\frac{2p^{4n} -p }{n^{2} + n + 1} + 2p^{2n}p^{i}}{4C} $ for all $i\in [n]$, 
		\item $\probs_{\initialstate, \alpha_{i}, w_{2}}=\frac{\frac{2p^{4n} -p }{n^{2} + n + 1} - 2p^{2n}p^{i}}{4C} $ for all $i\in [n]$, 
		\item $\probs_{\initialstate, \alpha_{i}, s^{+}_{\gamma}}=0$ for all $i\in [n]$, 
		\item $\probs_{\initialstate, \alpha_{i}, q}= 1 - \sum_{s \in Succ(\initialstate) \setminus \lbrace q \rbrace} \probs_{\initialstate, \alpha_{i}, s}$ for all $i\in [n]$, 
		\item $\probs_{\initialstate, \beta_{ij}, s^{-}_{k}}=\frac{1}{2n^{2} + 2n + 2}$ for all $i,j,k \in [n]$,
		\item $\probs_{\initialstate, \beta_{ij}, s^{+}_{k}}=0$ for all $i \in [n]$, 
		\item $\probs_{\initialstate, \beta_{ij}, v^{-}_{kl}}=\frac{1}{2n^{2} + 2n + 2}$ for all $i, j,k,l \in [n]$ and $(i,j) \neq (k, l)$.
		\item $\probs_{\initialstate, \beta_{ij}, v^{-}_{ij}}=0$ for all $i, j \in [n]$.
		\item $\probs_{\initialstate, \beta_{ij}, v^{+}_{ij}}=0$ for all $i, j \in [n]$.
		\item $\probs_{\initialstate, \beta_{ij}, w_{k}}=\frac{\frac{2p^{4n} -p }{n^{2} + n + 1} + p^{i+j}}{4C} $ for all $i,j\in [n]$ and $k \in [2]$
		\item $\probs_{\initialstate, \beta_{ij}, s^{+}_{\gamma}}=0$ for all $i,j\in [n]$, 
		\item $\probs_{\initialstate, \beta_{ij}, q}= 1 - \sum_{s \in Succ(\initialstate) \setminus \lbrace q \rbrace} \probs_{\initialstate, \beta_{ij}, s}$ for all $i,j\in [n]$,
		
		\item $\probs_{\initialstate, \gamma, s^{-}_{i}}=\frac{1}{2n^{2} + 2n + 2}$ for all $i \in [n]$,
		\item $\probs_{\initialstate, \gamma, s^{+}_{i}}=0$ for all $i \in [n]$, 
		\item $\probs_{\initialstate, \gamma, v^{-}_{ij}}=\frac{1}{2n^{2} + 2n + 2}$ for all $i,j \in [n]$, 
		\item $\probs_{\initialstate, \gamma, v^{+}_{ij}}=0$ for all $i,j \in [n]$, 
		\item $\probs_{\initialstate, \gamma, w_{i}}=\frac{2p^{4n} -p }{4C(n^{2} + n + 1)} $ for all $i\in [2]$, 
		\item $\probs_{\initialstate, \gamma, s}=0$ for all $s \in Succ(\initialstate) \setminus \lbrace w_{1}, w{2}, q \rbrace$,
		\item $\probs_{\initialstate, \gamma, q}=0$ for all $s \in Succ(\initialstate) \setminus \lbrace q \rbrace$,  
		\item $\probs_{\initialstate, \delta, w_{i}}=\frac{1}{2}$ for all $i\in [2]$, 
		\item $\probs_{\initialstate, \delta, s}=0$ for all $s \in Succ(\initialstate) \setminus \lbrace w_{1}, w{2} \rbrace$,
		\item $\probs_{s, \epsilon, s} = 1$ for all $s \in \states \setminus \initialstate \rbrace$, and
		\item $\probs_{s, \epsilon, z} = 1$ for all $s \in \states \setminus \initialstate \rbrace$ and $s \neq z$.
	\end{itemize}
	
	We note that all states that can be reached from the initial state $\initialstate$ are absorbing.

	We consider the following specifications for the supervisor
	\begin{enumerate}
		\item $\Pr^{\policy^{\svisor}}_{\mdp}(\initialstate \models \lozenge \lbrace s^{+}_{i} | i\in [n] \rbrace \cup \lbrace s^{+}_{ij} | i,j\in [n] \rbrace \cup \lbrace s^{+}_{\gamma} \rbrace ) \geq \frac{1}{2n^{2} + 2n + 2}$, \label{supspec1}
		\item $\Pr^{\policy^{\svisor}}_{\mdp}(\initialstate \models \lozenge \lbrace s^{-}_{i} \rbrace )\geq \frac{1}{2n^{2} + 2n + 2}\left(1 - \frac{1}{n^{2} + n + 1}\right)$ for all $i\in [n]$, \label{supspec2}
		\item $\Pr^{\policy^{\svisor}}_{\mdp}(\initialstate \models \lozenge \lbrace v^{-}_{ij} \rbrace) \geq \frac{1}{2n^{2} + 2n + 2}\left(1 - \frac{1}{n^{2} + n + 1}\right)$ for all $i,j\in [n]$, \label{supspec3}
		\item $\Pr^{\policy^{\svisor}}_{\mdp}(\initialstate \models \lozenge \lbrace s^{+}_{i}, v^{-}_{ij} \rbrace) \geq\frac{1}{2n^{2} + 2n + 2}$ for all $i,j\in [n]$ and $i\neq j$,	\label{supspec4}	
		\item $\Pr^{\policy^{\svisor}}_{\mdp}(\initialstate \models \lozenge \lbrace s^{+}_{j}, v^{-}_{ij} \rbrace) \geq \frac{1}{2n^{2} + 2n + 2}$ for all $i,j\in [n]$ and $i\neq j$, \label{supspec5}
		\item $\Pr^{\policy^{\svisor}}_{\mdp}(\initialstate \models \lozenge \lbrace v^{+}_{ij}, s^{-}_{i}, s^{-}_{j} \rbrace) \geq \frac{1}{2n^{2} + 2n + 2}$ for all $i,j\in [n]$ and $i\neq j$, \label{supspec6}
		\item $\Pr^{\policy^{\svisor}}_{\mdp}(\initialstate \models \lozenge \lbrace v^{+}_{ii}, s^{-}_{i} \rbrace) \geq \frac{1}{2n^{2} + 2n + 2}$ for all $i\in [n]$,\label{supspec7}
		\item $\Pr^{\policy^{\svisor}}_{\mdp}(\initialstate \models \lozenge \lbrace v^{-}_{ii}, s^{+}_{i} \rbrace) \geq \frac{1}{2n^{2} + 2n + 2}$ for all $i\in [n]$, \label{supspec8}
		\item $\Pr^{\policy^{\svisor}}_{\mdp}(\initialstate \models \lozenge \lbrace s^{+}_{j} | M_{ij} = 1 \rbrace) \geq \frac{1}{2n^{2}+2n + 2} \frac{1}{n^2 + n + 1}$, for all $i\in [m]$, and 
		\item  $\Pr^{\policy^{\svisor}}_{\mdp}(\initialstate \models \lozenge \lbrace s^{-}_{j} | M_{ij} = 1 \rbrace) \geq \frac{1}{2n^{2}+2n + 2} \left( d_{i} - \frac{1}{n^2 + n + 1} \right)$ for all $i\in [m]$. 
	\end{enumerate}
	
	We consider the specification for the agent $\Pr^{\policy^{\agent}}_{\mdp}(\initialstate \models \lozenge w_{1}) \geq 1/2$. We note that the agent's policy satisfies $\pi^{\agent}_{\initialstate, \delta} = 1$ since the other actions reach $w_{1}$ with probability less than $1/2$.

We note that the policies for $\mdp$ can only differ for state $\initialstate$ at time $0$ since an absorbing state is reached after the first transition. Therefore, we can without loss of generality we can assume that the reference policy is stationary. We have the optimization problem 
\small
	\begin{subequations}
		\begin{align}
		&\max \quad &&\frac{1}{2} \log\left( \frac{1}{2k_{1}}\right) + \frac{1}{2} \log\left( \frac{1}{2k_{2}}\right)
		\\
		&\normalfont \text{subject to} 
		\\
		&\forall a \in \actions(\initialstate) \label{positiveactions},  && \pi^{\svisor}_{\initialstate,a} \geq 0,
		\\
		& && \sum_{i=1}^{n} \pi^{\svisor}_{\initialstate, \alpha_{i}} + \sum_{i=1}^{n}\sum_{j=1}^{n}  \pi^{\svisor}_{\initialstate, \beta_{ij}} + \pi^{\svisor}_{\initialstate, \gamma} + \pi^{\svisor}_{\initialstate, \delta}  = 1, \label{addupto1}
		\\
		& && \frac{1}{2n^{2}+2n + 2} \left( \sum_{i=1}^{n} \pi^{\svisor}_{\initialstate, \alpha_{i}} + \sum_{i=1}^{n}\sum_{j=1}^{n}  \pi^{\svisor}_{\initialstate, \beta_{ij}} + \pi^{\svisor}_{\initialstate, \gamma}  \right) \geq \frac{1}{2n^{2}+2n + 2},	\label{supspeccons1}
		\\
		& \forall i\in [n], && \frac{1}{2n^{2}+2n + 2} \left( \sum_{\substack{j=1 \\ j \neq i}}^{n} \pi^{\svisor}_{\initialstate, \alpha_{j}} + \sum_{j=1}^{n} \sum_{k=1}^{n}  \pi^{\svisor}_{\initialstate, \beta_{jk}} +  \pi^{\svisor}_{\initialstate, \gamma} \right) \geq \frac{1}{2n^2+2n+2}\left(1 - \frac{1}{n^{2}+n + 1} \right), \label{supspeccons2}
		\\
		&\forall i,j\in [n], && \frac{1}{2n^{2}+2n + 2} \left( \sum_{\substack{k = 1}}^{n} \pi^{\svisor}_{\initialstate, \alpha_{k}} + \underset{  (i,j) \neq (k,l)}{\sum_{k=1}^{n} \sum_{l=1}^{n}} \pi^{\svisor}_{\initialstate, \beta_{kl}} +  \pi^{\svisor}_{\initialstate, \gamma} \right) \geq \frac{1}{2n^{2}+2n + 2} \left(1 - \frac{1}{n^2+n+1} \right), \label{supspeccons3}
		\\
		&  \forall i,j \in [n], && \frac{1}{2n^{2}+2n + 2} \left( \pi^{\svisor}_{\initialstate, \alpha_{i}} + \sum_{\substack{k = 1}}^{n} \pi^{\svisor}_{\initialstate, \alpha_{k}} + \underset{  (i,j) \neq (k,l)}{\sum_{k=1}^{n} \sum_{l=1}^{n}} \pi^{\svisor}_{\initialstate, \beta_{kl}} +  \pi^{\svisor}_{\initialstate, \gamma}\right) \geq \frac{1}{2n^{2}+2n + 2} , \label{supspeccons4}
		\\
		& \forall i,j \in [n], && \frac{1}{2n^{2}+2n + 2} \left( \pi^{\svisor}_{\initialstate, \alpha_{j}} + \sum_{\substack{k = 1}}^{n} \pi^{\svisor}_{\initialstate, \alpha_{k}} + \underset{  (i,j) \neq (k,l)}{\sum_{k=1}^{n} \sum_{l=1}^{n}} \pi^{\svisor}_{\initialstate, \beta_{kl}} +  \pi^{\svisor}_{\initialstate, \gamma} \right) \geq \frac{1}{2n^{2}+2n + 2}, \label{supspeccons5}
		\\
		&  \substack{\forall i,j \in [n] \\ i\neq j}, && \frac{1}{2n^{2}+2n + 2} \left( \pi^{\svisor}_{\initialstate, \beta_{ij}} + \sum_{\substack{k = 1 \\ k \neq i}}^{n} \pi^{\svisor}_{\initialstate, \alpha_{k}} +  \sum_{\substack{k = 1 \\ k \neq j}}^{n} \pi^{\svisor}_{\initialstate, \alpha_{k}} + 2 \sum_{k=1}^{n} \sum_{l=1}^{n} \pi^{\svisor}_{\initialstate, \beta_{kl}} +  2\pi^{\svisor}_{\initialstate, \gamma} \right) \geq \frac{1}{n^{2}+n + 1}  - \frac{1}{2(n^2 + n + 1)^2}, \label{supspeccons6}
		\\
		& \forall i \in [n], && \frac{1}{2n^{2}+2n + 2} \left(  \pi^{\svisor}_{\initialstate, \beta_{ii}} + \sum_{\substack{k = 1\\ k \neq i}}^{n} \pi^{\svisor}_{\initialstate, \alpha_{k}} + \sum_{k=1}^{n} \sum_{l=1}^{n} \pi^{\svisor}_{\initialstate, \beta_{kl}} +  \pi^{\svisor}_{\initialstate, \gamma} \right) \geq \frac{1}{2n^{2}+2n + 2}, \label{supspeccons7}
		\\
		&\forall i \in [n], && \frac{1}{2n^{2}+2n + 2} \left(  \pi^{\svisor}_{\initialstate, \alpha_{i}} + \sum_{\substack{k = 1 }}^{n} \pi^{\svisor}_{\initialstate, \alpha_{k}} + \underset{  (i,i) \neq (k,l)}{\sum_{k=1}^{n} \sum_{l=1}^{n}} \pi^{\svisor}_{\initialstate, \beta_{kl}} +  \pi^{\svisor}_{\initialstate, \gamma}\right) \geq \frac{1}{2n^{2}+2n + 2}, \label{supspeccons8}
		\\
		&\forall i \in [m], && \frac{1}{2n^{2}+2n + 2} \left(\sum_{\substack{j=1 \\ M_{ij} = 1}}^{n}   \pi^{\svisor}_{\initialstate, \alpha_{i}} \right) \geq \frac{1}{2n^{2}+2n + 2} \frac{1}{n^2 + n + 1}, \label{supspeccons9}
		\\
		& \forall i \in [m], &&  \frac{1}{2n^{2}+2n + 2} \left(\sum_{\substack{j=1 \\ M_{ij} = 1}}^{n} \left( \sum_{\substack{k=1 \\ k \neq j}}^{n} \pi^{\svisor}_{\initialstate, \alpha_{k}} + \sum_{k=1}^{n} \sum_{l=1}^{n}  \pi^{\svisor}_{\initialstate, \beta_{kl}} +  \pi^{\svisor}_{\initialstate, \gamma} \right) \right) \geq \frac{1}{2n^{2}+2n + 2} \left( d_{i} - \frac{1}{n^2 + n + 1} \right),\label{supspeccons10}
		\\
		& && k_{1} = \frac{2p^{4n} -p }{4C(n^{2} + n + 1)} \pi^{\svisor}_{\initialstate, \gamma} +\sum_{i=1}^{n} \frac{\frac{2p^{4n} -p }{n^{2} + n + 1} + 2p^{2n}p^{i}}{4C} \pi^{\svisor}_{\initialstate, \alpha_{i}} + \sum_{i=1}^{n}\sum_{j=1}^{n}\frac{\frac{2p^{4n} -p }{n^{2} + n + 1} + p^{i+j}}{4C} \pi^{\svisor}_{\initialstate, \beta_{ij}} + \frac{\pi^{\svisor}_{\initialstate, \delta}}{2}
				\\
		& && k_{2} = \frac{2p^{4n} -p }{4C(n^{2} + n + 1)} \pi^{\svisor}_{\initialstate, \gamma} +\sum_{i=1}^{n} \frac{\frac{2p^{4n} -p }{n^{2} + n + 1} + 2p^{2n}p^{i}}{4C} \pi^{\svisor}_{\initialstate, \alpha_{i}} + \sum_{i=1}^{n}\sum_{j=1}^{n}\frac{\frac{2p^{4n} -p }{n^{2} + n + 1} + p^{i+j}}{4C} \pi^{\svisor}_{\initialstate, \beta_{ij}} +  \frac{\pi^{\svisor}_{\initialstate, \delta}}{2}
		\end{align}
	\end{subequations}
\normalsize
	where the decision variables are $\pi^{\svisor}_{\initialstate, a}$ for all $a \in \actions(\initialstate)$, $k_{1}$, and $k_{2}$.  The constraints \eqref{positiveactions} and \eqref{addupto1} ensure the validity of the reference policy, and constraints \eqref{supspeccons1}-\eqref{supspeccons10} correspond to the specifications $1$-$10$, respectively. 

	\end{proof}

We note that $\pi^{\svisor}_{\initialstate, \delta} = 0$ due to constraints \eqref{positiveactions}, \eqref{addupto1}, and \eqref{supspeccons1}. Using $\pi^{\svisor}_{\initialstate, \delta} = 0$ and constraint \eqref{addupto1}, and scaling the constraints with $2n^2 + 2n + 2$, we get the following equivalent optimization problem
 
\small
\begin{subequations}
	\label{opt:scaledversion}
	\begin{align}
	&\max \quad &&\frac{1}{2} \log\left( \frac{1}{2k_{1}}\right) + \frac{1}{2} \log\left( \frac{1}{2k_{2}}\right)
	\\
	&\normalfont \text{subject to} 
	\\
	&\forall a \in \actions(\initialstate)\setminus\lbrace \gamma, \delta \rbrace, && \pi^{\svisor}_{\initialstate,a} \geq 0,
	\\
	& && 1 - \sum_{a \in  \actions(\initialstate)\setminus\lbrace \gamma, \delta \rbrace } \pi^{\svisor}_{\initialstate,a} \geq 0 \label{uselesscons}
	\\
	& \forall i\in [n], && \pi^{\svisor}_{\initialstate, \alpha_{i} } \leq  \frac{1}{n^{2}+n + 1}, \label{limitedalpha}
	\\
	&\forall i,j\in [n], && \pi^{\svisor}_{\initialstate, \beta_{ij}} \leq \frac{1}{n^2+n+1} , \label{limitedbeta}
	\\
	&  \forall i,j \in [n], &&  \pi^{\svisor}_{\initialstate, \alpha_{i}}  \geq \pi^{\svisor}_{\initialstate, \beta_{ij}} , 
	\\
&  \forall i,j \in [n], &&  \pi^{\svisor}_{\initialstate, \alpha_{j}}  \geq \pi^{\svisor}_{\initialstate, \beta_{ij}} , 
	\\
	&  \substack{\forall i,j \in [n] \\ i\neq j}, &&  \pi^{\svisor}_{\initialstate, \beta_{ij}}  \geq \pi^{\svisor}_{\initialstate, \alpha_{i}} + \pi^{\svisor}_{\initialstate, \alpha_{j}} - \frac{1}{(n^2 + n + 1)}, 
	\\
	& \forall i \in [n], &&  \pi^{\svisor}_{\initialstate, \beta_{ii}} \geq \pi^{\svisor}_{\initialstate, \alpha_{i}},
	\\
& \forall i \in [n], &&  \pi^{\svisor}_{\initialstate, \beta_{ii}} \leq \pi^{\svisor}_{\initialstate, \alpha_{i}},
	\\
	&\forall i \in [m], && \sum_{\substack{j=1 \\ M_{ij} = 1}}^{n} \pi^{\svisor}_{\initialstate, \alpha_{j}} \geq \frac{1}{n^2 + n + 1}, 
	\\
	& \forall i \in [m], &&  \sum_{\substack{j=1 \\ M_{ij} = 1}}^{n}  \pi^{\svisor}_{\initialstate, \alpha_{j}} \leq \frac{1}{n^2 + n + 1},
	\\
& && k_{1} = \frac{2p^{4n} -p }{4C(n^{2} + n + 1)} +\sum_{i=1}^{n} \frac{ 2p^{2n}p^{i}}{4C} \pi^{\svisor}_{\initialstate, \alpha_{i}} + \sum_{i=1}^{n}\sum_{j=1}^{n}\frac{ p^{i+j}}{4C} \pi^{\svisor}_{\initialstate, \beta_{ij}},
\\
& && k_{2} = \frac{2p^{4n} -p }{4C(n^{2} + n + 1)} +\sum_{i=1}^{n} \frac{ 2p^{2n}p^{i}}{4C} \pi^{\svisor}_{\initialstate, \alpha_{i}} + \sum_{i=1}^{n}\sum_{j=1}^{n}\frac{ p^{i+j}}{4C} \pi^{\svisor}_{\initialstate, \beta_{ij}},
	\end{align}
\end{subequations}
\normalsize
where the decision variables are $\pi^{\svisor}_{\initialstate, a}$ for all $a \in \actions(\initialstate)\setminus\lbrace \gamma, \delta \rbrace$, $k_{1}$, and $k_{2}$.

Note that \eqref{uselesscons} is trivially satisfied since $$1 - \sum_{a \in  \actions(\initialstate)\setminus\lbrace \gamma, \delta \rbrace } \pi^{\svisor}_{\initialstate,a} \geq 1 - \sum_{a \in  \actions(\initialstate)\setminus\lbrace \gamma, \delta \rbrace } \frac{1}{n^2 + n + 1} = 1 -  \frac{n^2 + n}{n^2 + n + 1} = \frac{1}{n^2 + n + 1} \geq 0$$ where the first inequality is due to \eqref{limitedalpha} and \eqref{limitedbeta}.

We define $x_{i}= \pi^{\svisor}{\initialstate, \alpha_{i}}(n^{2} + n + 1)$ for all $i \in [n]$ and $y_{ij} = \pi^{\svisor}{\initialstate, \beta_{ij}}(n^{2} + n + 1)$ for all $i,j \in [n]$. Using $x_{i}, y_{ij}$ variables, the optimization problem given in \eqref{opt:scaledversion} can be written as 

\small
\begin{subequations}
	\label{opt:finalverisonsupp}
	\begin{align}
	&\max \quad && \frac{1}{2} \log\left( \frac{1}{(2p^{4n} - p + 2p^{2n} x_{0} + y_{0}) (2p^{4n} - p - 2p^{2n} x_{0} + y_{0})}\right) + \frac{1}{2} \log\left( 4C^2(n^2 + n + 1)^2\right)
	\\
	&\normalfont \text{subject to} && x_{i} \geq 0, \ \forall i \in [n]
	\\
	& && y_{ij} \geq 0, \ \forall i ,j \in [n]
	\\
	& \forall i\in [n], && x_{i} \leq 1
	\\
	&\forall i,j\in [n], && y_{ij} \leq 1
	\\
	&  \substack{\forall i,j \in [n] \\ i\neq j}, && x_{i}  \geq y_{ij}, 
	\\
	&  \substack{\forall i,j \in [n] \\ i\neq j}, && x_{j}  \geq y_{ij} , 
	\\
	&  \substack{\forall i,j \in [n] \\ i\neq j}, &&  y_{ij}  \geq x_{i} + x_{j} - 1, 
	\\
	& \forall i \in [n], &&  y_{ii} = x_{i},
	\\
	&\forall i \in [m], && \sum_{\substack{j=1 \\ M_{ij} = 1}}^{n} x_{j} = 1, \label{consM2supp}
	\\
	& && x_{0} = \sum_{i=1}^{n} p^{i} x_{i}
	\\
	& && y_{0} = \sum_{i=1}^{n} \sum_{j=1}^{n} p^{i+j} y_{ij} 
	\end{align}
\end{subequations}
\normalsize
where the decision variables are $x_{i}$ for all $i \in [n]$ and $y_{ij}$ for all $i,j \in [n]$.

Due to the result given in \cite{matsui1996np}, the optimal value of \eqref{opt:finalverisonsupp} is is greater than or equal to $-\log(4p^{8n})/2 + \log\left( 4C^2(n^2 + n + 1)^2\right)/2 $ if and only if there exists a $0-1$ solution for $x_{1}, \ldots, x_{n}$ satisfying \eqref{consM2supp}. Since the decision problem of \eqref{opt:finalverisonsupp} correspond to solving the set partition problem, \eqref{opt:finalverisonsupp} is NP-hard.  

Since the number of states, actions, and the task constraints is polynomial in $n$ and \eqref{opt:finalverisonsupp} synthesizes an optimal reference policy, the synthesis of optimal reference policies is NP-hard.

\end{document}